\documentclass[a4paper,11pt,leqno,english]{amsart}
\usepackage[utf8]{inputenc}
\usepackage[T1]{fontenc}
\usepackage{microtype}
\usepackage[french, main=english]{babel}
\usepackage{amssymb,latexsym}

\usepackage{mathrsfs,tgschola}
\usepackage[normalem]{ulem}
\usepackage{pgfplots}
\usepackage{tkz-fct}
\usepackage{url}
\usepackage{xcolor}
\usepackage{comment}
\definecolor{violet}{rgb}{0.0,0.2,0.7}
\definecolor{rouge2}{rgb}{0.8,0.0,0.2}
\usepackage{tikz}
\usepackage{empheq}
\usepackage{tikz-cd}
\usetikzlibrary{matrix,arrows,decorations.pathmorphing}
\usepackage{hyperref}
\usepackage{mathpazo}
\usepackage{enumerate}
\usepackage{geometry}

\hypersetup{
    bookmarks=true,         % show bookmarks bar?
    unicode=false,          % non-Latin characters in Acrobatâs bookmarks
    pdftoolbar=true,        % show Acrobatâs toolbar?
    pdfmenubar=true,        % show Acrobatâs menu?
    pdffitwindow=false,     % window fit to page when opened
    pdfstartview={FitH},    % fits the width of the page to the window
    pdftitle={},    % title
    pdfauthor={},     % author
    colorlinks=true,       % false: boxed links; true: colored links
   linkcolor=rouge2,          % color of internal links
    citecolor=violet,        % color of links to bibliography
    filecolor=black,      % color of file links
    urlcolor=cyan}           % color of external links
\usepackage{enumitem}
\usepackage{appendix}

 \theoremstyle{plain}    
 \newtheorem{thm}{Theorem}[section]
\theoremstyle{plain} 
\theoremstyle{thm*} 
\newtheorem{bigthm}{Theorem}

 \numberwithin{equation}{section} %% Comment out for sequentially-numbered
 \numberwithin{figure}{section} %% Comment out for sequentially-numbered
 \newtheorem{cor}[thm]{Corollary} %%Delete [thm] to re-start numbering
 \theoremstyle{plain}    
 \newtheorem{prop}[thm]{Proposition} %%Delete [thm] to re-start numbering
 \theoremstyle{plain}    
  \theoremstyle{thm*}    
  \newtheorem*{question*}{Question}
 \newtheorem{lem}[thm]{Lemma} %%Delete [thm] to re-start numbering
 \theoremstyle{remark}
  \newtheorem{claim}[thm]{Claim} %%Delete [thm] to re-start numbering
 \theoremstyle{remark}
 \newtheorem{rem}[thm]{Remark}
 \theoremstyle{definition}
\newtheorem{exa}[thm]{Example}
\theoremstyle{thm*}
\newtheorem*{setup*}{Main Setup}
\theoremstyle{thm}
\newtheorem{setup}[thm]{Setup}
\theoremstyle{plain}

\theoremstyle{definition}

\newcommand{\C}{{\mathbb{C}}}

\newcommand{\Z}{{\mathbb{Z}}}

\newcommand{\be}{\mathbf e}

\newcommand{\parallelsum}{\mathbin{\!/\mkern-5mu/\!}}

\def\1{\bold{1}}

\newcommand{\wom}{\widehat{\Omega}}

\newcommand{\e}{\varepsilon}

\newcommand{\om}{\omega}

\newcommand{\ep}{\varepsilon}

\DeclareMathOperator{\Ric}{Ric}
\newcommand{\Hol}{\mathrm{Hol}}
\newcommand{\SU}{\mathrm{SU}}
\newcommand{\Sp}{\mathrm{Sp}}

\renewcommand{\ge}{\geqslant}
\renewcommand{\le}{\leqslant}

\title[On log Calabi-Yau manifolds]{Log Calabi--Yau manifolds: holomorphic tensors, stability and universal cover}

\author{Tristan C. Collins}
\address{Department of Mathematics, University of Toronto, 40 St. George St., Toronto, Ontario, Canada}
\email{tristanc@math.toronto.edu}

\author{Henri Guenancia}
\address{Univ. Bordeaux, CNRS, Bordeaux INP, IMB, UMR 5251, F-33400 Talence, France}
\email{henri.guenancia@math.cnrs.fr}

\date{}

\begin{document}

\begin{abstract}
We study various geometric properties of log Calabi-Yau manifolds, i.e. log smooth pairs $(X,D)$ such that $K_X+D=0$. More specifically, we focus on the two cases where $X$ is a Fano manifold and $D$ is either smooth or has two proportional components.  Despite the existence of a complete Ricci flat Kähler metric on $X\setminus D$ in both cases, we will show that the geometric properties of the pair $(X,D)$ are vastly different, e.g. validity of Bochner principle, local triviality of the quasi-Albanese map, polystability of $T_X(-\log D)$ and compactifiability of the universal cover of $X\setminus D$. When $D$ has two components we show that the universal cover of $X\setminus D$ is a Calabi-Yau manifold of infinite topological type, and we describe the geometry at infinity from a Riemannian point of view.

\end{abstract}

\maketitle
\tableofcontents

\section{Introduction}

\subsection{Ricci flat Kähler metrics and applications: a brief summary}
Let $X$ be a compact Kähler manifold such that $c_1(X)=0\in H^2(X, \mathbb R)$. A seminal result of Yau \cite{Yau78} shows the existence in any given Kähler class of a unique Kähler metric $\omega$ which is Ricci flat, i.e. $\Ric \om =0$. This result unlocked a host of applications towards our understanding of the geometry of compact Kähler manifold with zero first Chern class, including: 
\begin{enumerate}
\item {\it Bochner principle:} holomorphic tensors are parallel with respect to $\omega$. 
\item {\it Stability:} the tangent bundle $T_X$ is polystable with respect to any Kähler class. 
\item {\it Albanese:} the Albanese map $\mathrm{alb}_X: X\to \mathrm{Alb}(X)$ is a locally trivial fibration, which becomes trivial after a finite base change. 
\item {\it Universal cover:} the universal cover $\widetilde X$ is isomorphic to $\mathbb C^r\times Y$ where $r\in \mathbb N$ and $Y$ is compact Kähler. In particular, $\widetilde X$ is an analytic Zariski open set in a compact Kähler manifold. 
\item {\it Beauville-Bogomolov decomposition:} there is a a finite étale cover $X'\to X$ which splits as a product $X'\simeq T\times \prod_{i\in I} Y_i\times \prod_{j\in J} Z_j$ where $T$ is a complex torus, the $X_i$ are irreducible Calabi-Yau manifolds and the $Y_j$ are irreducible holomorphic symplectic manifolds. 
\end{enumerate}

Note that the Beauville-Bogomolov decomposition theorem \cite{Bea83} encompasses $(1)-(4)$ in the sense that the first four properties can easily be derived from Yau's theorem and the said decomposition. It quickly became important to look for larger classes of Kähler manifolds with zero first Chern class that could be endowed with Kähler Ricci flat metrics and for whom results like $(1)-(5)$ might be desirable to obtain. At least two such classes stand out.

\medskip

One of these classes consists in compact Kähler varieties with log terminal singularities and zero first Chern class; a class of varieties arising naturally in the context of the Minimal Model Program, see e.g. \cite{Pe94}. The existence of singular Kähler Ricci flat metrics was proved by \cite{EGZ} (later complemented by \cite{Paun} in the Kähler case), Bochner principle was established in \cite{GGK} and \cite{CGGN}, polystability of the tangent was proved in \cite{GKP} and \cite{GSS}, étale triviality of the Albanese map (in the canonical singularity case) was proved in \cite{Kawa81} and \cite{CGGN}. As for the universal cover either of $X$ or $X_{\rm reg}$, it is conjectured to have a similar structure but so far only partial results are known, cf e.g. \cite{GGK}. Finally, the decomposition theorem was proved in the algebraic case in several steps \cite{Dru16,GGK,HP} and in the Kähler case by \cite{BGL}, relying on the algebraic case. 

Very recently, it was showed in \cite{BFPT} that none of these results can be extended in the setting of varieties with log canonical singularities, as they produced such a variety where the Albanese map is not locally trivial (in a strong sense) and the tangent sheaf is not polystable.  However, it is expected that in such a setting, no natural Kähler Ricci flat metric can be constructed on the regular locus of the variety. 

\medskip

Another important class consists in the complement of a divisor $D$ in a compact Kähler manifold $X$ such that $K_X+\alpha D\sim_{\mathbb Q} 0$ for some $\alpha \in \mathbb Q_{>0}$. The construction of a {\it complete} Kähler Ricci flat metric on $X\setminus D$ was achieved by Tian and Yau when $X$ is Fano, $D$ is smooth and either $\alpha=1$ \cite{TY1} or $\alpha >1$ \cite{TY2}. When $\alpha \in (0,1)$ and $D$ has simple normal crossings, {\it incomplete} Kähler Ricci flat metrics on $X\setminus D$ with conic singularities along $D$ were constructed by \cite{Brendle}, \cite{CGP, GP} and \cite{JMR} in various degrees of generality, and applications like Bochner principle or stability of the (orbifold) tangent bundle were derived in \cite{CGP, GP} and \cite{GT16}. Finally, in the case where $X$ is Fano, $D=D_1+D_2$ has simple normal crossings with two numerically proportional irreducible components and $\alpha=1$, Collins and Li \cite{CL} constructed a {\it complete} Kähler Ricci flat metric on $X\setminus D$. 

\subsection{Main results}

The goal of this paper is threefold. 
\begin{enumerate}[label=$\bullet$]
\item Provide algebro-geometric applications of the existence of the Tian-Yau metric on the complement of a smooth anticanonical divisor in a Fano manifold, e.g. by showing that the suitable analogue of properties $(1)-(4)$ are satisfied. 
\item Exhibit an infinite family of examples of log Calabi-Yau pairs $(X, D=D_1+D_2)$ where $X$ is Fano and $D$ is an anticanonical divisor with two irreducible components and {\it all four} properties $(1)-(4)$ fail, despite the existence of a complete Kähler Ricci flat metric on $X\setminus D$. 
\item Study the asymptotic geometry of the universal cover of $(X\setminus(D_1\cup D_2),\omega_{\rm CL})$. 
\end{enumerate}

\medskip

In the following, we let $X$ be a Fano manifold of dimension $n\ge 2$ and we let $D\in |-K_X|$ be an anticanonical divisor with simple normal crossings. We set $M:=X\setminus D$. We focus on the following two cases. 

\bigskip

{\bf Case 1.  $D$ smooth.} 

\noindent
In that case, it is not hard to check that the quasi-Albanese variety of $(X,D)$ is a point (e.g. case $k=p=1$ in Proposition~\ref{easy vanishing}) and that the complement $M=X\setminus D$ has finite fundamental group, so that $(3)$ and $(4)$ are trivially true. It remains to investigate $(1)$ and $(2)$. Let $\omega_{\rm TY}$ be the complete Tian-Yau metric on $M$. 
%As $\pi_1(M)$ is finite, it is not too difficult to derive from the asymptotics of $\omega_{\rm TY}$ that its holonomy has to be $\mathrm{SU}(n)$, cf Proposition~\ref{hol one}. 

\begin{bigthm}
\label{thm A}
Let $X$ be a Fano manifold of dimension $n\ge 2$, let $D\in |-K_X|$ be a smooth anticanonical divisor and let $\omega_{\rm TY}$ be the complete Tian-Yau metric on $X\setminus D$. Finally, let $p,q$ be nonnegative integers, and let $E:=T_X(-\log D)^{\otimes p}\otimes \Omega_X(\log D)^{\otimes q}$ be a logarithmic tensor bundle.  
\begin{enumerate}
\item {\bf [Holonomy]} The holonomy of $(M, \omega_{\rm TY})$ is $\SU(n)$. 
\item {\bf [Bochner principle].} Any holomorphic tensor $\sigma \in H^0(X, E)$ is parallel with respect to $\omega_{\rm TY}$ on $X\setminus D$.
\item {\bf [Stability].} The logarithmic tangent bundle $T_X(-\log D)$ is stable with respect to $-K_X$. 
\end{enumerate}
\end{bigthm}

Let us give a few comments about each of these two statements. 

\smallskip

{\bf A.1.} The computation of the holonomy of $\omega_{\rm TY}$ relies essentially on two facts. First, the asymptotics of $\omega_{\rm TY}$ forces the vanishing of any parallel $p$-form ($p\neq 0,n$) on $M$ or on any finite étale cover. Then, finiteness of $\pi_1(M)$ combined with the usual results on holonomy (de Rham splitting, Berger-Simons classification) allows to conclude, cf Proposition~\ref{hol one}.

\smallskip

{\bf A.2.} Bochner principle can be rephrased in the following slightly more precise way. If $x\in X\setminus D$ and $V:=E_x\simeq (\mathbb C^n)^{\otimes p}\otimes ((\mathbb C^n)^*)^{\otimes q}$, then the evaluation map at $x$, $\mathrm{ev}_x: H^0(X, E) \longrightarrow V$, induces a isomorphism onto the subspace $V^{\SU(n)}\subset V$ consisting of vectors which are invariant under the standard representation of $\SU(n)$ on $V$, cf Theorem~\ref{bochner}. In particular, parallel transport of a vector $v\in V^{\SU(n)}$ yields a holomorphic section of $E|_{X\setminus D}$ which extends automatically across $D$ (i.e. the covariant part acquires zeros and the contravariant part acquires at most logarithmic poles). 

A key technical reason for the validity of the Bochner principle is the fact that $\omega_{\rm TY}$ has {\it subquadratic volume growth}. 

\smallskip

{\bf A.3.} Here again, one can refine the statement a bit by saying that the vector bundle $E$ in the theorem is polystable with respect to $-K_X$ and that its decomposition into stable summands $E=\bigoplus E_i$ is given by the decomposition of the semisimple representation $V$ of $\SU(n)$ (cf item 1 above). Finally, in restriction to $X\setminus D$ and with respect to $\omega_{\rm TY}$, the decomposition above is orthogonal and each summand $E_i$ is parallel, cf Theorem~\ref{stability}. 

One can think of the stability of $T_X(-\log D)$ as a type of Donaldson-Uhlenbeck-Yau correspondence given that the hermitian metric induced by $\omega_{\rm TY}$ on $T_M$ is Hermitian-Einstein. Such a correspondence has been developed in a "Tian-Yau like" context by J. Zhang \cite{JZhang24}. Although his results do not cover ours, there is a common flavor in some of the computations. 

Stability is proved with respect to the class $-K_X=D$ which plays a very special role. One could wonder if stability holds for any polarization as in compact case. 

\bigskip

{\bf Case 2.  $D=D_1+D_2$.} 

\noindent
In that case, we assume additionally that $n\ge 3$ and that $D_1$ and $D_2$ are proportional (hence they are proportional to $-K_X$). In particular, one can consider the complete Ricci flat Kähler metric $\omega_{\rm CL}$ on $M$ constructed by Collins and Li. A first important difference with the case $D$ smooth is that $M$ has infinite fundamental group (see Corollary~\ref{cor pi1}) and $H^0(X, \Omega_X^1(\log D))\neq 0$.

Actually, the latter space is one-dimensional and one can find an explicit generator $\alpha$ (cf Lemma~\ref{qalb}). From the asymptotics of $\omega_{\rm CL}$, it is easy to see that $\alpha|_M$ cannot be parallel, hence Bochner principle fails, cf Remark~\ref{bochner fails}. Note that the volume growth of $\omega_{\rm CL}$ is superquadratic (it is of order $R^{\frac{4n}{n+2}}$) in contrast to $\omega_{\rm TY}$, which is one way to explain the discrepancy between the two cases studied here from the point of view of Bochner principle. 

Let us now move on to the properties $(2)-(4)$ mentioned at the beginning of the introduction. 

\begin{bigthm}
\label{thm B}
Let $n\ge 3$ be an odd integer and let $D_1, D_2$ be general hypersurfaces of degree $\frac{n+1}{2}$. Set $X:=\mathbb P^n$ and $D:=D_1+D_2$.
\begin{enumerate} 
\item The quasi-Albanese map $\mathrm{alb}_{(X,D)}:X\setminus D \to \mathrm{Alb}(X,D)$ is not locally trivial. 
\item The logarithmic tangent bundle $T_X(-\log D)$ is not polystable.
\item The universal cover of $X\setminus D$ is not biholomorphic to an analytic Zariski open subset of a compact complex manifold.  
\end{enumerate}
\end{bigthm}

Here again, let us give a few comments about each of the three statements above. 

\smallskip

{\bf B.1.} The first item is rather easy to prove once one has identified $\mathrm{alb}_{(X,D)}$ with a Lefschetz pencil (see Example~\ref{pencil} and the discussion above it), since such a pencil necessarily has some singular fibers. However, one can significantly strengthen the statement when $n\ge 9$ and $D_1, D_2$ are very general, in which case two very general fibers of $\mathrm{alb}_{(X,D)}$ are not birational to each other. This follows from deep results on stable irrationality due to Schreieder \cite{Schreieder} as showed by Shinder \cite{Shinder}, cf Proposition~\ref{prop:alb}. Another example of log smooth Calabi-Yau pair with non isotrivial quasi Albanese map was provided in \cite{BFPT}.

\smallskip

%The second item is a consequence of the first one by the same arguments as in the appendix of \cite{BFPT} by Müller, cf Proposition~\ref{prop:alb}. 
{\bf B.2.}  Note that $T_X(-\log D)$ is however semistable with respect to any polarization for any log canonical pair $(X,D)$ with $D$ reduced and $K_X+D\equiv 0$ by \cite{GSS}.

It is also worth noting that properties $(1)-(2)$ remain true for any finite étale cover of $\mathbb P^n\setminus D$, cf Proposition~\ref{prop:alb}.

The polystability of $T_X(-\log D)$ fails in the more general setting where $X$ is Fano, $D=\sum_{i=1}^kD_i$ SNC with $k<n$ and $D_i\sim D_1$ for any index $i$. Moreover, one can determine the Jordan-Hölder filtration in that case, cf the relevant paragraph on page~\pageref{JH}. 

\smallskip

{\bf B.3.} This result is to be compared with Yau's compactification conjecture \cite{YauProb, YauNonlinear} , which asserts that complete Calabi-Yau manifolds homotopic to a finite $CW$-complex can be compactified by adding a multiple of an anti-canonical divisor.  Anderson-Kronheimer-Lebrun \cite{AKL} produced examples of complete hyperK\"ahler $4$-manifolds of type ``$A_{\infty}$" for which the second homology $H_2$ is infinitely generated, and the intersection matrix on $H_2$ is determined by the Dynkin digram $A_{\infty}$.  The construction of \cite{AKL}  uses an infinite Gibbons-Hawking construction. Hattori \cite{Hat11} computed the volume growth of these examples.  Goto \cite{Goto94} constructed hyperK\"ahler manifolds of type $A_{\infty}$ in dimension $4m$ for all $m$, and subsequently \cite{Goto98} constructed examples of $4$-dimensional hyperK\"ahler manifolds of type $A_{\infty}$ and $D_{\infty}$ using an infinite dimensional hyperK\"ahler quotient construction.  Note that our examples are not hyperK\"ahler, cf. Theorem~\ref{thm C} below. In our case, the middle homology group of the universal cover of $X\setminus D$ has infinite rank, cf Theorem~\ref{univ cover} and Remark~\ref{rem univ cover}.

%This statement has to be compared with Yau's compactification conjecture \cite{YauProb, YauNonlinear} and the two-dimensional example constructed by \cite{AKL}. In our case, the middle homology group of the universal cover of $X\setminus D$ has infinite rank, cf Theorem~\ref{univ cover} and Remark~\ref{rem univ cover}. 

In a nutshell, one could think about that last property as follows. The quasi-Albanese map $\mathrm{alb}_{(X,D)}:M\to \mathbb C^*$ has a least a singular fiber with one ordinary double point, say over $1\in \mathbb C^*$. The base change by the exponential map $e^{2\pi i \cdot}:\mathbb C\to \mathbb C^*$ yields an infinite cover $\widehat M\to M$ endowed with a fibration $\widehat f:\widehat M\to \mathbb C$ where the fibers over $\mathbb Z$ have a singular point and are permuted by the Galois group $\mathbb Z$ of the cover. Now, one can connect the two singular points of $\widehat M_0:=\widehat f^{-1}(0)$ and $\widehat M_{1}:=\widehat f^{-1}(1)$ using the vanishing cycle in the nearby smooth fibers, cf Figure~\ref{fig 1} on page \pageref{fig 1}. The upshot is that this construction yields an actual $n$-cycle $L$ fibering over the segment $[0, 1]$ and that the Galois translates of $L$ are linearly independent in $H_n(\widehat M, \mathbb Z)$. 
\bigskip

Our last result deals with the asymptotic geometry of the universal cover of $M$ equipped with $\omega_{\rm CL}$. Recall that there exists a convex function $u$ on $(\mathbb R_{>0})^2$, homogeneous of degree $\frac{n+2}{n}$ in the variable $x_1+x_2$ such that the asymptotic tangent cone of $(M,\omega_{\rm CL})$ is $((\mathbb R_{>0})^2, g_{\infty})$ where 
\[g_{\infty}=\sum_{i,j=1}^2 u_{ij}(x) dx_i\otimes dx_j\] 
and $u_{ij}=\partial^2_{ij} u$. Moreover, the volume growth of a ball of radius $R$ in $(M,\omega_{\rm CL})$ is of order $R^{\frac{4n}{n+2}}$. 

\begin{bigthm}
\label{thm C}
Let $X$ be a Fano manifold of dimension $n\ge 3$, let $D=D_1+D_2$ be an anticanonical divisor with simple normal crossings such that $\mathcal O_X(D_1)$ and $\mathcal O_X(D_2)$ are proportional and let $\omega_{\rm CL}$ be the complete Ricci flat Kähler metric on $M=X\setminus D$.
\begin{enumerate}
\item The holonomy of $(M, \omega_{\rm CL})$ is $\SU(n)$. 
\item The asymptotic tangent cone of the universal cover $(\widetilde M, \widetilde \omega_{\rm CL})$ is $(\mathbb R_{>0})^2\times \mathbb R$ with a metric of the form $g_{\infty}+h(x_1, x_2)dt^2$.
\item The volume growth of a ball of radius $R$ in $(\widetilde M,\widetilde \omega_{\rm CL})$ is of order $R^{\frac{6n}{n+2}}$. 
\end{enumerate}
\end{bigthm}
\smallskip

{\bf C.1.} Compared to the Tian-Yau case, the difficulty is that the fundamental group is infinite, hence the usual arguments based on the Berger-Simons classification only allow to determine the restricted holonomy group. After some reductions, one is left to eliminate the possibility that the holonomy is an infinite (discrete) extension of the symplectic group. 

Combined with Theorem~\ref{thm B}, this shows that there is no clear analogue of what an irreducible Calabi-Yau manifold should be in the non-compact case: should we use holonomy? the algebra of holomorphic differential forms? some stability condition on the logarithmic tangent bundle? And if a good notion is found, is there a hope to get a version of the decomposition result?

\smallskip

{\bf C.2-3.} As explained before, $\pi_1(M)\simeq \mathbb Z\times G$ with $G$ abelian finite. A generator of the $\mathbb Z$ factor is represented by a certain loop around $D$ which gets collapsed by $\omega_{\rm CL}$ at infinity. In the universal cover $\widetilde M$, the loop unwraps and yields a line which survives in the asymptotic cone of $\widetilde\omega_{\rm CL}$ and gives the factor $dt^2$. It also contributes to the change of volume growth from $R^{\frac{4n}{n+2}}$ in $M$ to $R^{\frac{6n}{n+2}}$ in $\widetilde M$.

\bigskip

{\bf Acknowledgements. } The authors are grateful to J\'anos Koll\'ar for interesting exchanges and for explaining to us the argument we reproduced in Remark~\ref{rem kollar}. H.G. would also like to thank Junsheng Zhang for a nice and useful discussion about his result \cite{JZhang24}. 

\medskip

{\bf Funding.} Part of this work was supported by the National Science Foundation under Grant No. DMS-1928930 while H.G. was in residence at the Simons Laufer Mathematical Science Institute (former MSRI) in Berkeley, California, during the Fall 2024 semester. H.G. is partially supported by the French Agence Nationale de la Recherche (ANR) under reference ANR-21-CE40-0010 (KARMAPOLIS).  T.C.C. is supported in part by NSERC Discovery grant RGPIN-2024-518857, and NSF CAREER grant DMS-1944952.  The authors are supported by the France-Canada Research Fund.

\section{Preliminaries}

\begin{setup*}
\label{set gen}
Let $X$ be a Fano manifold of dimension $n\ge 2$ and let $D=\sum_{i=1}^k D_i$ be a reduced divisor with simple normal crossings such that $K_X+D= 0$ as Cartier divisors. We set $M:=X\setminus D$ and fix a trivializing section $\Omega\in H^0(X,K_X+D)$ which we will freely view as a holomorphic $n$-form on $M$ with logarithmic poles along $D$. 
\end{setup*}

Note that since $D$ is ample and $n\ge 2$, Lefschetz theorem implies that $D$ is connected.

\subsection{Topological considerations}
We begin with two well-known facts.

\begin{lem}
\label{lem fg}
In the Main Setup, assume that each component $D_i$ is ample. Then there is an exact sequence 
\[H_2(X,\mathbb Z)\overset{\delta}{\longrightarrow} \mathbb Z^k\longrightarrow \pi_1(X\setminus D)\longrightarrow 1\]
where $\delta(C)= (\int_C c_1(D_i))_{i=1, \ldots, k}$.
\end{lem}

\begin{proof}
There are two ingredients. First, since each component $D_i$ is ample, a result of Nori \cite[Corollary 2.10]{Nori} implies that $\pi_1(X\setminus D)$ is isomorphic to the fundamental group of $E^\times$ where $E$ is the total space of the bundle $\mathcal O_X(D_1) \oplus \cdots\oplus \mathcal O_X(D_k)$ and $E^\times$ is the complement in $E$ of $\bigcup_{i=1}^k \ker p_i$ where $p_i$ is the projection $E\to \mathcal O_X(D_i)$.

Next, $E^\times$ retracts onto an $(S^1)^k$-bundle over $X$ and since $X$ is simply connected, Hurewicz's theorem implies that the map $\pi_2(X)\to H_2(X,\mathbb Z)$ is isomorphic. The exact sequence in the lemma is provided by the long exact sequence in homotopy associated to the fiber bundle $E^\times \to X$ and the description of the boundary map $\delta$ can be obtained as in the proof of \cite[Proposition~2]{Shimada} which deals with the case of a single component. 
\end{proof}

\begin{cor}
\label{cor pi1}
Under the assumptions of Lemma~\ref{lem fg} above, we have
\begin{enumerate}
\item The  fundamental group $\pi_1(X\setminus D)$ is abelian.
\item If $k=1$, $\pi_1(X\setminus D)$ is finite. 
\item If $k=2$, $\pi_1(X\setminus D)$ is infinite if, and only if $D_1$ and $D_2$ are proportional in $H^2(X,\mathbb R)$. In that case, $\pi_1(X\setminus D)$ is a finite extension of $\mathbb Z$. 
%Moreover, the free part of $\pi_1(X\setminus D)$ is generated by a loop wrapping once around $D_1$ and then around $D_2$ in the reverse direction. 
\end{enumerate}
\end{cor}
\begin{proof}
Only the third item needs some explanation. The map $H_2(X,\mathbb Z)\to \mathbb Z^k$ is nontrivial since $D$ is ample, hence has finite cokernel if $k=1$. If $k=2$, the cokernel is finite if, and only if there is no non-trivial relation between $D_1$ and $D_2$ in the real Neron-Severi group.
% As for the last sentence in $(3)$, recall the more precise statement of Nori's theorem, yielding that the map $X\setminus D\to E^\times$ provided by the respective canonical sections of $D_i$ induces an isomorphism at the level of fundamental group. Then it is clear that the image of $(1,-1)\in \mathbb Z^2$
\end{proof}

\subsection{Description of the \'etale covers}
\label{sec:covers}
Let us now give a geometric interpretation to the finite covers of $X\setminus D$ in the cases $k=1, 2$. 

\begin{lem}
\label{torsion}
In the Main Setup, assume additionally that $k=1$ and let $d$ be the index of $\mathcal O_X(D)$, or equivalently $K_X$.  Then the universal cover of $X\setminus D$ is induced by the co-restriction to the complement of $D$ of the cyclic cover corresponding the $d$-th root of $\mathcal O_X(D)$.
\end{lem}

\begin{proof}
Let $m\in \mathbb Z_{>0}$ be such that the image of $\delta:H_2(X, \mathbb Z)\to \mathbb Z$ is $m\mathbb Z$. We will explain below that $m=d$, i.e. $\mathcal O_X(D)$ has index exactly $m$, from which the lemma follows. 

\noindent
Let us now justify our claim. Since $X$ is Fano, the first Chern class induces an isomorphism $\mathrm{Pic}(X)\overset{\sim}{\longrightarrow} H^2(X,\mathbb Z)$. Since $H_1(X, \mathbb Z)$ is trivial, the universal coefficient theorem shows that there is no torsion in $H^2(X, \mathbb Z)$ and that the canonical map $\Phi:H^2(X,\mathbb Z)\to H_2(X, \mathbb Z)^\vee$ is an isomorphism, thanks to Poincaré duality. Our assumption is that $\delta=\Phi(c_1(D)) \in H_2(X, \mathbb Z)^\vee$ is divisible by $m$ but no other larger integer, hence so is $c_1(D)$. This implies that $D$ has index $m$. 
\end{proof}

Let us now investigate the case where $k=2$. 

\begin{lem}
In the Main Setup, assume additionally that $k=1$ and that $D_1$ is linearly equivalent to $D_2$. Let $d$ be the index of $D_1$ (or $D_2$, equivalently). Then $\pi_1(M)\simeq \mathbb Z\times \mathbb Z/d\mathbb Z$.
\end{lem} 

\begin{proof} From the proof of Lemma~\ref{lem fg}, it follows that $\pi_1(M)$ is the quotient of $\mathbb Z^2$ by the subgroup $H:=\{(du,du), u\in \mathbb Z\} \simeq \mathbb Z$.  Now, the morphism 
\[\phi:\mathbb Z^2\longrightarrow \mathbb Z \times \mathbb Z_{d}\]
defined by $\phi(u,v)=(u-v, \bar u)$ is easily seen to be surjective with kernel $H$. This shows the lemma. 
\end{proof}

In what follows, we will describe geometrically the Galois covers of $M$. 
\smallskip

$\bullet$ {\it The covers $\widehat M$ and $M_\ell$.}
Since $D_1\sim_{\mathbb Z} D_2$, there is a map $f:M\to \mathbb C^*$, cf Section~\ref{sec:alb}. In particular, one can produce étale covers of $M$ by performing base change for the fibration $f$. There are two types of such base change, coming from the exponential map $\mathbb C\to\mathbb C^*$ or the map $\mathbb C^*\to \mathbb C^*$ given by $z\mapsto z^\ell$, $\ell \in \mathbb Z$. We set $\widehat M:=M\times_{\mathbb C^*} \mathbb C$ and $M_\ell:=M\times_{\mathbb C^*}\mathbb C^*$ the respective covers. Set $G:=\mathbb Z\times \mathbb Z_d\simeq \pi_1(M)$; under the Galois correspondence, $\widehat M$ and $M_\ell$ correspond respectively to the kernels of the maps $\mathrm{pr}_1:G\to \mathbb Z $ and its reduction modulo $m$, $\overline{\mathrm{pr_1}}^\ell :G\to \mathbb Z_\ell$. 

\smallskip

$\bullet$ {\it The cover $N_1$.} 
The finite étale cover corresponding to the kernel of the second projection  $\mathrm{pr}_2:G\to \mathbb Z_d$ can be interpreted geometrically as the degree $d$ cyclic cover $N_1\to M$ induced by the $d$-torsion line bundle $\mathcal O_X(D_1)|_{M}$. It is not hard to see that the universal cover $\widetilde M$ of $M$ is isomorphic to $N_1\times_M \widehat M$.

\smallskip

$\bullet$ {\it The cover $N_{12}$.} 
Another geometrically relevant cover of $M$ is the cyclic cover $N_{12}\to M$ induced by the $2d$-torsion line bundle $\mathcal O_X(D_1+D_2)|_{M}\sim_{\mathbb Z}  \mathcal \mathcal O_X(-K_X)|_{M}$, which we will now relate to the covers mentioned above and describe from the group theoretic vantage point. \\

For $i=1,2$, we let $s_i$ be a section of $\mathcal O_X(D_i)$ whose divisor of zeros is $D_i$, and we denote by $L$ the bundle such that $dL\sim_{\mathbb Z} D_i$.  Recall that one can choose $f$ to coincide with $\frac{s_1}{s_2}$ and that $N_1$ can be described as $\{v\in L; v^d=s_1\}$, while $N_{12}$ can be identified with $\{v\in L; v^{2d}=s_1s_2\}$, and the maps $N_1\to M$ and $N_{12}\to M$ are induced by the natural projection $p_L:L\to X$. We have the following diagram
\[
\begin{tikzcd}
N_1\times_M M_{2d} \arrow[rr, "{{\rm pr}_2}"] \arrow[dr, "\alpha"] \arrow[dd, "{{\rm pr}_1}", swap] && M_{2d} \arrow[d]  \\
&N_{12}  \arrow[r, "\beta"] \arrow[dr] &M_2 \arrow[d]\\
N_1 \arrow[rr]&&M
\end{tikzcd}
\]
where the maps $\alpha$ and $\beta$ are given by $\alpha(v,w)=\frac vw$ and $\beta(v)=(p_L(v), \frac{v^d}{s_2})$ where $v\in L, w\in \mathbb C^*$ and the variable in $M$ is omitted. Let 
$t:\Z_d\to \Z_{2d}$ be the injection induced by multiplication by $2$ on $\Z$. It follows from the description of the cover morphism $\alpha$ that the latter corresponds to the group morphism $\Z_d\times \Z_{2d}\to \Z_{2d}$ given by $(\bar k,\bar \ell)\mapsto t(\bar{k})-\bar \ell$ and $N_{12}\to M$ is given by the kernel of the map $G\to \Z_{2d}$ defined by $(\ell, \bar k)\mapsto t(\bar k)-\bar \ell$. 

\subsection{Elementary vanishing theorems}
The following proposition is likely to be well-known to the experts, but we add it here for lack of a proper reference. 
\begin{prop}
\label{easy vanishing}
In the Main Setup, assume that each component $D_i$ is ample. Then for any $k \le p < n$, we have 
\[H^0(X,\Omega^p_X(\log D))=0.\]
\end{prop}

\begin{proof}
We argue by induction on $k$. 
\medskip

\emph{Case $k=1, p=1$.}

\noindent
The residue $\sigma_D$ of $\sigma$ along $D$ is an element in  $H^0(D,\mathcal O_{D})$, hence a constant. Moreover, one has $\sigma_D [D]=0 \in H^2(X, \mathbb C)$ by the residue theorem, cf e.g. \cite[Theorem~3.1]{Per}. This implies that $\sigma$ is actually holomorphic but since $X$ is Fano one has $H^0(X, \Omega_X^1)=0$ hence $\sigma \equiv 0$ as desired. \\
% By a result of Deligne, we have an non-canonical isomorphism
%\begin{equation}
%\label{deligne}
%H^1(X\setminus D,\mathbb C) =H^0(X,\Omega^1_X(\log D))\oplus H^1(X,\mathcal O_X).
%\end{equation}
% By Corollary~\ref{cor pi1}, Hurewicz's theorem and the universal coefficients theorem,  $H^1(X\setminus D,\mathbb C)=0$, which shows the claim for $p=1$.\\
 
\emph{Case $k=1, p\ge 2$}.  

\noindent
Let $\sigma \in H^0(X,\Omega_X^p(\log D))$. The residue $\sigma_D$ of $\sigma$ along $D$ is an element in  $H^0(D,\Omega^{p-1}_{D})$. By Lefschetz theorem for Hodge groups, the restriction map 
 \[H^0(X,\Omega^{p-1}_{X})\to H^0(D,\Omega^{p-1}_{D})\]
is an isomorphism for $p\le n-1$ but the LHS vanishes by Kodaira's vanishing theorem since $X$ is Fano. We infer that $\sigma_D=0$, hence $\sigma$ is actually an element in $H^0(X,\Omega_X^p)$ but the latter space is zero by Kodaira theorem again. \\

\emph{Case $k>1$}.   

\noindent
Let $\sigma \in H^0(X,\Omega_X^p(\log D))$ with $k \le p < n$. The residue $\sigma_{D_i}$ of $\sigma$ along $D_i$ is an element in  $H^0(D_i,\Omega^{p-1}_{D_i}(\log\Delta_i))$ where $\Delta_i:=(\sum_{j\neq i} D_j)|_{D_i}$ has $k-1$ components. Now, by adjunction we have $K_{D_i}+\Delta_i=0$. In particular $D_i$ is Fano of dimension $n-1$ and we can apply induction since $k-1\le p-1<n-1$ so that $\sigma_{D_i}=0$ and $\sigma$ is an element in $H^0(X,\Omega_X^p)$. We can now conclude as in the previous case. 
\end{proof}

\begin{rem}
The result above is sharp. Indeed, assume that each $D_i$ is $\mathbb Q$-proportional to $-K_X$ so that there exist positive integers $d, d_1, \ldots, d_k$ such that $-dK_X= d_iD_i$ for all indices $i$. Set $f_i:=\frac{s_i^{\otimes d_i}}{s_k^{\otimes d_k}}$ for $1\le i \le k-1$ and consider the form $\sigma_p:=d \log f_1 \wedge \cdots \wedge d \log f_{p}$ for $1\le p \le k-1$. It yields a logarithmic $p$-form $\sigma_p\in H^0(X,\Omega^p(\log D))$. Since $D$ is snc and $k\le n-1$, the intersection $W:=D_1\cap \ldots \cap D_{k}$ has positive dimension and near a given point of $W$ one can find a system of holomorphic coordinates such that $D_i=(z_i=0)$ and 
\[\sigma_{k-1}= \sum_{i=1}^{k} (-1)^{k-i}d_1 \cdots \hat d_i\cdots d_k   \frac{dz_1\wedge \ldots \wedge \widehat {dz_i} \wedge \cdots \wedge d z_{k}}{z_1 \cdots \hat z_i \cdots  z_{k}} \] is not zero (hence $\sigma_p$ is also non-zero for $1\le p \le k-1$). 
\end{rem}

The following result is classical (see e.g. \cite[Theorem~2.8]{SW}) when $D$ is smooth.  
\begin{prop}
\label{no vf}
In the Main Setup, assume that $k<n$ and that each component $D_i$ is ample. Then we have 
\[H^0(X,T_X(-\log D))=0.\]
\end{prop}

\begin{proof}
We proceed by induction on $k$, with $k=1$ being known already. So we assume that $k\ge 2$. Fix an index $i$ and set $\Delta_i:=(\sum_{j\neq i} D_j)|_{D_i}$. Note that $D_j|_{D_i}$ is irreducible (and smooth) since $k<n$ and each $D_i$ is ample, therefore $\Delta_i$ has $k-1<n-1$ components. Moreover $-K_{D_i}\sim \Delta_i$ is ample by assumption.   

Now, let $\xi \in H^0(X,T_X(-\log D))$; the restriction of $\xi$ to $D_i$ yields an element in $\xi_i\in H^0(D_i,T_{D_i}(-\log \Delta_i))$. By induction, $\xi_i=0$ hence $\xi \in H^0(X,T_X\otimes  \mathcal O_X(-D)))$. Since $T_X\simeq \Omega_X^{n-1}\otimes K_X^{-1}$, the latter space is isomorphic to $H^0(X,\Omega_X^{n-1})$ which vanishes by Kodaira theorem, since $X$ is Fano.
\end{proof}

\begin{rem}
\label{p2 lq}
Proposition~\ref{no vf} does not hold if $k= n$ as one can see by considering $(\mathbb P^2, L+Q)$ where $L$ is a line and $Q$ is a quadric. In the second case, the action of $\mathrm{SO}(2,\mathbb C)$ yields a non-zero logarithmic vector field. E.g., one can identify $M:=\mathbb P^2\setminus(L\cup Q)$ with $\mathbb C^2\setminus (x^2+y^2+1=0)$ and $\xi:=y\frac{\partial}{\partial x}-x\frac{\partial}{\partial y}$ is such a vector field. Moreover, $\xi$ is vertical for the fibration $f: M\to \mathbb C^*$ given by $f(x,y)=x^2+y^2+1$ and extends as a section of $T_{\mathbb P^2/\mathbb P^1}$ for the rational fibration $\overline f:\mathbb P^2\dashrightarrow \mathbb P^1$ which is well-defined away from $L\cap Q=\{[1:i:0],[1,-i:0]\}$ and whose fibers over $0$ and $\infty$ are respectively $Q$ and $2L$. In particular, $\xi \in H^0(\mathbb P^2, T_{\mathbb P^2}(-\log(L+Q)))$.

The proof of Proposition~\ref{no vf} fails in this example as restricting the pair to the line yields the new pair $(\mathbb P^1, \{0\}+\{\infty\})$ which on which lives the logarithmic vector field $z\frac{\partial}{\partial z}$. 
\end{rem}

\section{Case of one divisor}

In this section, we work in the following setup. 

\begin{setup}
\label{set one}
In the Main Setup, we assume additionally that $D$ is irreducible (hence smooth) and consider the complete Ricci flat metric $\omega$ on $M$ constructed by Tian-Yau. We denote by $\nabla$ the induced Chern connection on $T_{M}$.
\end{setup}

\subsection{Holonomy}
\begin{prop}
\label{hol one}
In Setup~\ref{set one}, one has $\Hol(M, \omega)=\SU(n)$.
\end{prop}

\begin{proof}
Here, all norms and covariant derivatives are taken with respect to $\omega$. We fix a base point $x_0\in M$ and set $G:=\Hol(M, \omega)$. The logarithmic volume form $\Omega$ has constant norm hence it is parallel by Bochner formula $\Delta |\Omega|^2=|\nabla \Omega|^2$. Therefore $G\subset \SU(n)$. 

Consider the universal cover $(\widetilde M, \widetilde \omega)$, which comes equipped with a finite \'etale Galois map $\pi^\circ:\widetilde M\to M$ thanks to Corollary~\ref{cor pi1}. This map induces an identification of $\Hol(\widetilde M, \widetilde \omega)$ with the identity component $G^\circ$ of $G$.

Therefore one can extend the latter map to a finite map $\pi:Y\to X$ between normal projective varieties ramifying over $D$ (e.g. since $X$ is simply connected), say at order $m\ge 2$. Set $\widetilde D=\pi^{-1}(D)$ so that  $K_{Y}=\pi^*(K_X+(1-\frac 1m)D)$. In particular, $Y$ is $\mathbb Q$-Fano (i.e. $Y$ is klt and and $-K_{Y}$ is ample). Actually, we know by Lemma~\ref{torsion} that $Y\to X$ is simply the cyclic cover associated to a $d$-th root of $\mathcal O_X(D)$, hence $Y$ is smooth. However, we do not need to know a priori that $Y$ is smooth in the argument below, which we will use again in a more general context cf Claim~\ref{no p forms 2}.  

Since $Y$ is normal, one can find at a general point $\widetilde x \in \widetilde D$ local holomorphic coordinates $(w_i)$ centered at $\widetilde x$ such that $\widetilde D=(w_1=0)$ (resp. coordinates $(z_i)$ centered at $\pi(x)$ such that $D=(z_1=0)$) such that $\pi(w_1,\ldots, w_n)=(w_1^m, w_2, \ldots, w_n)$. Moreover, $\omega$ is quasi-isometric to the model metric
\begin{equation}
\label{asymp}
\frac{idz_1\wedge d\bar z_1}{|z_1|^2 (-\log |z_1|)^{1-\frac 1n}}+(-\log |z_1|)^{\frac 1n}\sum_{k\ge 2} idz_k\wedge d\bar z_k.
\end{equation}
This shows that the asymptotics of $\widetilde \omega$ near a general point of $\widetilde D$ are the same as those of $\omega$ near $D$. 

\begin{claim}
\label{no p forms}
Let $0<p<n$ be an integer and let $\sigma \in \mathcal C^\infty(\widetilde M, \Omega_{\widetilde M}^{p,0})$ be a parallel $(p,0)$-form (i.e. $\nabla \sigma =0$). Then $\sigma \equiv 0$.
\end{claim}

\begin{proof}[Proof of Claim~\ref{no p forms}]
One can view $\sigma$ as a holomorphic section over $\widetilde M$ of the holomorphic vector bundle $\Omega_{Y}^p(\log \widetilde D)|_{\widetilde M}$. At a general point $\widetilde x\in \widetilde D$, one can write 
\[\sigma = \sum_{|J|=p-1} f_J \frac{dw_1}{w_1}\wedge dw_J+\sum_{|K|=p,\, 1\not \in K} f_K  dw_K.\]
The functions $f_J, f_K$ are defined and holomorphic away from $\widetilde D$. Moreover, it follows from \eqref{asymp} that one has
\begin{equation}
\label{norm}
|\sigma|^2 \approx \sum_J |f_J|^2 (-\log |w_1|)^{\frac {n-p}n}+\sum_K |f_K|^2 (-\log |w_1|)^{-\frac qn}.
\end{equation}
Since $\sigma$ is parallel, its norm is constant. In particular, $f_J, f_K$ are $L^2$ in a neighborhood of $\widetilde D$ hence extend holomorphically across $\widetilde D$. Since $|\sigma|$ is bounded and $p<n$, the functions $f_J$ have to vanish along $\widetilde D$. Then, $|\sigma|$ goes to zero near $\widetilde D$, hence $|\sigma|$ has to vanish everywhere. 
\end{proof}

We can now conclude easily. If the holonomy of $(\widetilde M, \widetilde \omega)$ were strictly included in $\SU(n)$, it would split as a nontrivial product $G^\circ=\prod G_i$ with $G_i\subset \SU(n_i)$ for some $n_i<n$. Parallel transport would yield a non-zero parallel $(0,n_i)$-form on $\widetilde M$, contradicting Claim~\ref{no p forms}. Therefore $G^\circ=\SU(n)$ and the result follows since $G^\circ \subset G\subset \SU(n)$.
\end{proof}

\subsection{Bochner principle}

Let $p,q\ge 0$ be two integers and let 
\[E:=T_X(-\log D)^{\otimes p}\otimes \Omega_X(-\log D)^{\otimes q},\] it is a holomorphic vector bundle on $X$. We also consider the $\mathscr C^\infty$ vector bundle $\mathcal E= (T_{M}^{1,0})^{\otimes p}\otimes (\Omega_{M}^{1,0})^{\otimes q}$ on $M$. We denote by $\mathscr C^\infty_{\parallelsum}(M,  \mathcal E)$ the space of global parallel sections of $\mathcal E$.

\begin{thm}[Bochner principle]
\label{bochner}
In Setup~\ref{set one}, the natural restriction map induces a one-to-one correspondence  
\[H^0(X,E)\overset{\sim}{\longrightarrow} \mathscr C^\infty_{\parallelsum}(M,  \mathcal E).\]

\noindent
 In particular, this space is isomorphic to  $(V^{\otimes p} \otimes (V^*)^{\otimes q})^G$ where $V:=\mathbb C^n$ is the standard representation and $G:=\SU(n)$.
\end{thm}

\begin{rem}
Bochner principle allows one to recover the case $k=1$ of Proposition~\ref{easy vanishing}.
\end{rem}

\begin{proof}
First, we claim that an element $\sigma\in \mathscr C^\infty_{\parallelsum}(M,  \mathcal E)$ extends across $D$ as holomorphic section of $D$; this has essentially already been established in the proof of Claim~\ref{no p forms}.  Indeed, one can locally near $D$ decompose any parallel section using a local holomorphic trivialization of $E$; since $\sigma$ has bounded norm the holomorphic functions appearing as coefficients have logarithmic growth along $D$ hence extend. \\

Conversely, let $\sigma \in H^0(X,E)$. If $(z_i)$ is a local system of holomorphic coordinates where $D=(z_1=0)$, then we have $|\sigma|^2 = O((-\log |z_1|)^N)$. Since the radius $r$ to a fixed point in $M$ grows like $(-\log |z_1|)^\frac{n+1}{2n}$, we get $|\sigma|^2=O(r^N)$, up to increasing $N$. 
Let $f:=\log (|\sigma|^2+1)$; we claim that the following inequality
\begin{equation}
\label{sh}
\Delta f \ge 0
\end{equation}
holds. Indeed, direct computation yields on $M$
\[i\partial \bar \partial f = \Big(\frac{|\nabla \sigma|^2}{| \sigma|^2+1}-\frac{|\langle \nabla \sigma, \sigma\rangle|^2}{(| \sigma|^2+1)^2}\Big)+\frac{1}{| \sigma|^2+1} \langle \Theta(T_X^{\otimes p}\otimes \Omega_X^{\otimes q},\omega)\sigma,\sigma\rangle.\]
The first term inside the parenthesis of the RHS is nonnegative by Cauchy-Schwarz. Since $\Ric \om=0$, the trace of the second term vanishes so that \eqref{sh} follows. \\

Let $R>0$ and let $\chi_R$ be a cut-off function with support in $B(2R)$, identically equal to $1$ on $B(R)$, satisfying 
\begin{equation}
\label{lap cutoff}
\Delta \chi_R=O(R^{-2}).
\end{equation} 
Such a function can be constructed as follows. Let $\xi: \mathbb R\to [0,1]$ be a non-increasing smooth function which is identically equal to $1$ (resp. $0$) on $(-\infty, 1]$ (resp. on $[2,+\infty)$). Next, let $s$ be a section of $\mathcal O_X(D)$ such that $(s=0)=D$ and let $h$ be a smooth hermitian metric on $\mathcal O_X(D)$. Set $\chi_R:=\xi(\frac 1R \cdot \varphi)$ where $\varphi:=(-\log |s|_h^2)^{\frac{n+1}{2n}}$. Since $r\approx \varphi$, $\chi_R$ satisfies the support requirements maybe up to replacing $B(R)\subset B(2R)$ by $B(C^{-1}R)\subset B(CR)$ for some universal constant $C>0$. Let us now check \eqref{lap cutoff}. Given the formula $dd^c \chi_R=\frac 1{R^2}(\xi'' d\varphi \wedge d^c \varphi+R\xi'dd^c \varphi)$, all we have left to do is checking that the forms $d\varphi \wedge d^c \varphi$ and $R dd^c \varphi$ are bounded with respect to $\omega$ on $B(2R)\setminus B(R)$. Set $\gamma(t)=t^{\frac{n+1}{2n}}$ for $t>0$. Then $\varphi = \gamma (-\log |s|_h^2)$ hence $d\varphi \wedge d^c \varphi= \gamma'^2 \frac{\langle \nabla s, \nabla s\rangle}{|s|^2}$ and $dd^c \varphi=  \gamma'' \frac{\langle \nabla s, \nabla s\rangle}{|s|^2}+\gamma' \Theta_h$. Setting $t:=-\log |s|$, then $R\simeq t^{\frac{n+1}{2n}}$ hence $R\gamma''(t) \simeq \gamma'^2(t) \simeq t^{\frac 1n-1}$ and $R\gamma'(t) \simeq t^{\frac 1n}$ on $B(2R)\setminus B(R)$ as $t\to +\infty$. The claim now follows from \eqref{asymp}.\\

Finally, since $f=O(\log R)$, $f \geq 0$ and $\Delta f \geq 0$ and  $\mathrm{vol}(B(R))  \simeq R^{\frac {2n}{n+1}}$, it follows from \eqref{lap cutoff} that
\[
2\int_{X}\chi_{R} |\nabla f|^2 \leq\int_X \chi_R \Delta f^2= \int_{B(2R)}f^2 \Delta \chi_R =O((\log R)^2 R^{-\frac{2}{n+1}}) \underset{R\to+\infty}{\longrightarrow} 0.
\]
It follows that $|\sigma|$ is constant. Since $\Delta |\sigma|^2=|\nabla \sigma|^2$, we infer that $\sigma$ is parallel on $M$.
\end{proof}

\subsection{Stability of the logarithmic tangent bundle}

For a log canonical pair $(X,D)$ with $D$ reduced and $K_X+D\sim \mathcal O_X$, it is known that the logarithmic tangent sheaf $T_X(-\log D)$ is semistable with respect to any polarization, cf \cite{GSS}. One may ask about the polystability (or stability). From a differential geometric point of view, the question can be approached once one has a canonical K\"ahler metric at hand. In the case where $X$ is a smooth Fano manifold and $D\in |-K_X|$ is smooth, we can work with the complete Ricci flat Tian-Yau metric $\omega$ on $M$. Since it has infinite volume, it is a priori not obvious that one can compute slopes of subsheaves of $T_X(-\log D)$  using $\omega$. 

Positive results in that direction have been previously obtained. For example, \cite{FS22} prove that $T_X(-\log D)$ is stable when $X=\mathbb P^n$, even allowing some isolated singularities for $D$. In
\cite{JZhang24}, J. Zhang proves a Donaldson-Uhlenbeck-Yau type theorem in that setting, under some assumptions for the asymptotics of the relevant metrics. Some of the computations in the result below have a similar flavor to those in \cite{JZhang24}.
%It is not clear to us that his result applies to our setting though. Indeed, the hermitian metric on $T_X(-\log D)$ induced by $\omega$ is not conformally smooth; moreover, the polystability of the bundle $T_X(-\log D)|_D$ is not obvious. Of course, as an extension of $T_D$ by $\mathcal O_D$ it is semistable, but the splitting of the extension may not be true (in analogy with the exact sequence $0\to T_D\to T_X|_D\to N_D\to 0$ which does not split when $X=\mathbb P^n$ and $D$ is a hypersurface of degree at least two, cf \cite{VdV}). 

\begin{thm}
\label{stability}
In Setup~\ref{set one}, $T_X(-\log D)$ is stable with respect to $D$. 
\end{thm}

\begin{rem}
\label{rem poly}
We actually get the following  stronger statement. The induced decomposition of any tensor bundle (e.g. $S^pT_X(-\log D)$, $\Omega_X(\log D)^{\otimes q}$, etc.) into stable summands is determined by the decomposition into irreducible factors of the associated tensor representation ($S^p V, (V^*)^{\otimes q}$, etc.) of the standard representation $V:=\mathbb C^n$ of $\SU(n)$.  
\end{rem}

\begin{proof}
Set $E:=T_X(-\log D)$. Since $E$ is semistable, it is enough to prove that there is no proper subsheaf $F\subset E$ with zero slope with respect to $D$. If one can show that such a subsheaf induces a parallel subbundle on $M$, then we would get a reduction of the holonomy to a proper subgroup of $\SU(n)$, in contradiction with Theorem~\ref{hol one}.\\

{\bf Step 1.} \emph{Moving the singularities away from $D$.}

\medskip

\noindent
So let us consider such a subsheaf $F\subset E$ with $c_1(F)\cdot D^{n-1}=0$ and $\mathrm{rk}(F)\neq 0,n$. Since saturation increases slope, one can assume that $F\subset E$ is saturated; in particular it is a subbundle away from a proper analytic set $Z=Z(\mathcal F)\subset X$ with codimension at least two. A crucial observation for the analysis to come is the following.
\begin{claim}
\label{claim subbundle}
One can assume that $Z(\mathcal F)\cap D=\emptyset.$
\end{claim}
\begin{proof}[Proof of Claim~\ref{claim subbundle}]
Consider the vector bundle $T_X(-\log D)|_D$. It sits in the exact sequence 
\[0\to T_D\to T_X(-\log D)|_D\to \mathcal O_D\to 0\]
which splits. Indeed, if $(z_1=0)$ is a local equation of $D$, we can consider the local logarithmic vector field $\xi:=z_1\frac{\partial}{\partial z_1}$ and it is straightforward to check that $\xi|_D$ is independent of the choice of coordinates hence it induces the desired section of $T_X(-\log D)|_D$. Another way to see $\xi$ is as the generator of the $\mathbb C^*$ action on the normal bundle. Since $D$ has trivial canonical bundle, it admits a Ricci flat K\"ahler metric hence its tangent bundle is polystable with respect to any polarization. Therefore,  $E|_D=T_X(-\log D)|_D$ is polystable with respect to $D|_D$. Now, $F|_D$ is a subsheaf of $E|_D$ satisfying  
\[c_1(F|_D)\cdot (D|_D)^{n-2} = c_1(F)\cdot D^{n-1}=0\]
hence its saturation $G_D\subset E|{_D}$ is a direct summand of $E|_D$. In particular, it is a subbundle of $E|_D$. It is elementary to construct a subbundle $G\subset E$ defined in a Zariski neighborhood of $D$ which restricts to $G_D$ on $D$. One can then extend it to a coherent saturated subsheaf of $E$ over the whole $X$, which we still denote by $G$. By construction, the slope of $G$ with respect to $D$ vanishes. This ends the proof of the claim.
\end{proof}

\medskip

{\bf Step 2.} \emph{Singularities of the metric induced on $\det F$ by $\omega$.}

\medskip

\noindent
We set $r:=\mathrm{rk}(F)$ and $L:=\det F$, which comes equipped with a injection of sheaves $j:L\to \Lambda^r E$ which is a morphism of vector bundles away from $Z$. 
%Note that $Z\cap D$ has codimension at least two in $D$. Otherwise, the saturation of $F|_D$ in $E|_D$ would have
Let us fix a smooth hermitian metric $h_L$ on $L$ and a smooth hermitian metric $h_D$ on $N_D=\mathcal O_X(D)|_D$ such that its curvature form $\omega_D$ is Ricci flat.  Since $c_1(F)\cdot D^{n-1}=c_1(F)|_D\cdot (D|_D)^{n-2}=0$, we have 
\begin{equation}
\label{zero slope}
\int_D \Theta_L(L, h_L) \wedge \omega_D^{n-2}=0.
\end{equation}

\medskip

%The pull back $j^*\Lambda^r h_E$ is a singular metric on $L$ of the form $h_Le^{\psi}$ where $\psi\in L^1(X)$ has log poles along $Z$, cf e.g. \cite[p. 165]{Koba}. We fix a log resolution $\pi:\widetilde X\to X$ such that $\pi^*dd^c \psi=[\Gamma]-\theta_\Gamma$ where $\theta_\Gamma$ is a smooth closed $(1,1)$-form and $\Gamma$ is a $\R_{>0}$-combination of $\pi$-exceptional divisors. \\

The Tian-Yau metric $\omega=dd^c \phi_{\rm TY}$ on $M$ induces a singular hermitian metric $h_E^\omega$ on $E$ satisfying $h_E^\omega=h_E e^s$ where $h_E$ is a smooth hermitian metric on $E$ and $s$ is a (singular) section of $\mathrm{End} E$. Pulling back by $j$, we get  a singular metric $h_L^\omega$ defined by
\[h_L^\omega:=j^*(\Lambda^r h_E^\omega)=:h_Le^{\psi_L}.\]
The latter metric develops singularities coming from the degeneracy of $j$ near $Z$ and the singular behavior of $s$ near $D$. Let us analyse $h_L^\omega$ more precisely. We work locally near a given point where $(e_1, \ldots, e_n)$ is a holomorphic frame for $E$. We get a frame $(e_I)_{|I|=r}$ for $\Lambda^r E$ where one can decompose $\Lambda^r e^s=(a_{IJ})$ and $\Lambda^r h_E=(h_{I\bar J})$. If $e_L$ is a local trivializing section of $L$, then $j(e_L)=\sum_I f_I e_I$ where $(f_I)$ are holomorphic such that their common zero set is supported on $Z$. In particular, we get 
\[|e_L|^2_{h_L^\omega}=\sum_{I,J,K,L}h_{J\bar L}f_I\overline f_K a_{IJ} \overline a_{KL}.\]
Therefore, we can decompose
\begin{equation}
\label{psi L}
\psi_L=\log |j|^2+\psi_s
\end{equation}
where the norm of $j$ is taken with respect to $\Lambda^r h_E\otimes h_L^{-1}$ and $\psi_s$ is such that
\begin{equation}
\label{psi s estimate}
|\psi_s| \le C \sup_{|v|=1} \log |(\Lambda^r e^s) \cdot v|
\end{equation}
for some constant $C>0$, where $v\in \Lambda^r E$ and norms are taken with respect to the smooth background metric $\Lambda^r h_E$. Note that $\log |j|^2$ extends to a quasi-psh function on $X$; since its polar set $Z$ has codimension two or more, it follows from Siu's decomposition theorem that the quasi-positive current $dd^c \log |j|^2$ put no mass on $Z$. \\

%In order to estimate the (rough) asymptotics of $s$, let us recall some well-known facts about the Tian-Yau metric. 

On the total space of $N_D$, we introduce the function $t=-\log |v|^2_{h_D}$ where $v$ is the fiber coordinate. Set $\phi_0:=\frac{n}{n+1}t^{1+\frac 1n}$ and 
\begin{equation}
\label{omega 0}
\omega_0:=dd^c \phi_0=\frac 1n t^{\frac 1n-1} \, dt \wedge d^ct+t^{\frac 1n }\,\omega_D
\end{equation}
to be the Ricci flat metric near the zero section obtained by the Calabi Ansatz. Then, there exists a diffeomorphism $F:V\to U$ from a neighborhood $U$ of the zero section in $N_D$ to a neighborhood $V$ of $D\subset X$ which is the identity on $D$ and satisfies
\begin{equation}
\label{asymp TY}
|\nabla^\ell_{\omega_0} (F^*\phi_{\rm TY}-\phi_0)|_{\omega_0}=O(e^{-\delta t})
\end{equation}
for some $\delta>0$, and for any $\ell \ge 0$, cf \cite{HSVZ}. In particular, one easily sees from the above asymptotics (equivalently  \eqref{asymp}) and \eqref{psi s estimate} that
\begin{equation}
\label{s est}
F^*\psi_s=O(\log t)
\end{equation}
holds on $U$ (recalling that $E=T_X(-\log D)$). It will be convenient to view the function $t$ as a function on $V$, which we extend by $1$ on $X\setminus V$.

\medskip

{\bf Step 3.} \emph{Computation of the slope of $F$ using $\omega$.}

\medskip

\noindent
On $M \setminus Z$, one has
\begin{equation}
\label{griffiths}
\Theta(L,h_L^{\omega})=\mathrm{pr}_L \Theta(\Lambda^r E, \Lambda^r h_E^{\omega})|_{L}-i\beta\wedge \beta^*
\end{equation}
where $\mathrm{pr}_L$ is the orthogonal projection with respect to $\Lambda^r h_E^{\omega}$ and $\beta$ is the second fundamental form associated to the $C^\infty$ splitting $\Lambda^r E=L\oplus L^\perp$; it is a $(1,0)$-form with values in $\mathrm{Hom}(L^\perp, L)$. Since $\omega$ is Ricci flat, we have $\Theta(\Lambda^r E, \Lambda^r h_E^{\omega})\wedge \omega^{n-1}=0$. From \eqref{psi L}, we have
\begin{equation}
\label{curv decomp}
\Theta(L,h_L^{\omega})=\Theta(L,h_L)-dd^c \log |j|^2-dd^c \psi_s.
\end{equation}
We next aim to integrate the three terms in the RHS separately. \\

We let $R>0$ be a large number and let $\chi=\chi_R(t)$ be a cut-off function supported on $(t\le 2R)$ which is equal to $1$ on $(t\le R)$. Since $dd^c \log |j|^2$ put no mass on $Z$, we have
\begin{eqnarray}
\label{IBP}
\int_{X\setminus Z} \chi dd^c \log |j|^2 \wedge \omega^{n-1}&=&\int_{X} \chi dd^c \log |j|^2 \wedge \omega^{n-1}\\
&=&\int_{X}  \log |j|^2 dd^c \chi \wedge \omega^{n-1}. \nonumber
\end{eqnarray}
Thanks to Step 1, the term $\log |j|^2$ is uniformly bounded on the support of $dd^c \chi$ at least when $R$ is large enough. Without this property, it is not clear to us that the analysis could be carried on. An elementary computation shows that $\int_{X} \chi dd^c \log |j|^2 \wedge \omega^{n-1}$ is controlled by 
\[ \int_{R\le t \le 2R } \big( |\chi' \phi_0''(\phi_0')^{n-2}|+|\chi'' (\phi_0')^{n-1}|\big) \, dt\wedge d^c t \wedge \omega_D^{n-1}=O(R^{-\frac 1n}).\]
In particular, we get that 
\begin{equation}
\label{sing j}
\int_{X\setminus Z} \chi dd^c \log |j|^2 \wedge \omega^{n-1}=O(R^{-\frac 1n}).
\end{equation}
Similarly, we get from \eqref{s est}
\begin{equation}
\label{IBP 2}
\int_X \chi dd^c\psi_s \wedge \omega^{n-1}=\int_{X} \psi_s dd^c \chi \wedge \omega^{n-1}=O\big( (\log R)\cdot R^{-\frac 1n}\big).
\end{equation}
\medskip

We now work on the support of $d\chi$, i.e. $(R\le t \le 2R)\subset V$ and we freely use the diffeomorphism $F$ to transport our objects to the normal bundle. We call $p:N_D\to D$ the projection map. It is convenient to set $\theta_D:=i\Theta(L,h_L)|_D$. Note that since $\Theta(L,h_L)$ is smooth, we have 
\begin{equation}
\label{theta L}
\Theta(L,h_L)-p^*\theta_D=O(e^{-\delta t})
\end{equation}
with respect to $\omega_0$ for any $\delta<\frac 12$, given \eqref{omega 0}. Since $\omega=dd^c \phi_{\rm TY}$, Stokes theorem yields
\[\int_{X} \chi \Theta(L,h_L) \wedge \omega^{n-1}=-\int_{X} d\chi \wedge d^c \phi_{\rm TY} \wedge  \Theta(L,h_L) \wedge \omega^{n-2}\]
and the latter is nothing but
\[\int_{R\le t \le 2R} \chi' \phi_0'dt\wedge d^c t\wedge p^*\theta_D \wedge \omega_0^{n-2}+O(e^{-\delta R}) \int_{R\le t \le 2R} \omega_0^n \]
for some $\delta>0$ thanks to \eqref{asymp TY} and \eqref{theta L}. 
Since $\chi, \phi_0$ are functions of $t$ only and $dt\wedge d^c t\wedge p^*\theta_D \wedge \omega_0^{n-2} = (\phi_0')^{n-2} dt\wedge d^c t\wedge p^*\theta_D \wedge \omega_D^{n-2}$, the vanishing \eqref{zero slope} implies that 
\[\int_{R\le t \le 2R} \chi' \phi_0'dt\wedge d^c t\wedge p^*\theta_D \wedge \omega_0^{n-2}=0.\]
Therefore, we find 
\begin{equation}
\label{slope 2}
\int_{X} \chi \Theta(L,h_L) \wedge \omega^{n-1}=O(e^{-\delta R})
\end{equation}
for some possibly smaller $\delta>0$. 

\noindent
Putting together \eqref{griffiths}, \eqref{curv decomp}, \eqref{sing j}, \eqref{IBP 2} and \eqref{slope 2}, we find
\[\lim_{R\to +\infty}\int_{X\setminus Z} \chi_R \, i \beta\wedge\beta^*\wedge \omega^{n-1} =0.\]
This shows that $(L,h_{L}^\omega)$ is parallel with respect to $\nabla$ on $M \setminus Z$. Basic linear algebra implies that $F_{X\setminus Z}$ is a parallel subbundle of $E$ over $M\setminus Z$, hence we get a holomorphic orthogonal splitting
\[E=F\oplus F^{\perp} \quad \mbox{on} \, M \setminus Z.\]
Since $F$ is reflexive, this decomposition extends to a direct sum decomposition of reflexive sheaves on $M$; in particular $F$ is a direct summand of $E$ on $M$ so it is locally free there. Therefore $Z=\emptyset$ and $F|_{M}$ is a parallel subbundle of $T_{M}$. We have reached the desired contradiction in view of Theorem~\ref{hol one}.

\medskip

{\bf Step 4.} \emph{Case of tensor bundles.}

\medskip

\noindent
Regarding the statement in Remark~\ref{rem poly}, we argue as follows. Since $E$ is a stable vector bundle on a projective manifold, it admits a Hermite-Einstein metric with respect to any fixed K\"ahler metric in $c_1(X)$. Therefore any tensor bundle $E'$ is Hermite-Einstein as well, hence polystable. Its stable summands are exactly the irreducible subbundles with zero slope. But one can run the same argument as for $E$ to show that any such subbundle has to be parallel with respect to the metric induced by $\omega$ on $M$, and therefore comes from an irreducible piece of the corresponding tensor representation for the holonomy group $\SU(n)$. Conversely, any parallel subbundle over $M$ extends to a subbundle on $X$ by the proof of Theorem~\ref{bochner}; it is a direct summand because one can argue similarly for its orthogonal complement.
\end{proof}

\section{Case of two divisors}

In this section, we work in the following setup. 

\begin{setup}
\label{set two}
In the Main Setup, we assume additionally that $D=D_1+D_2$ has exactly two irreducible components satisfying $d_iD_i=-dK_X$ with $d_i>0$ mutually prime.  We consider the  complete Ricci-flat metric $\omega$ on the complement $M:=X\setminus D$ constructed by Collins-Li \cite{CL}. Finally, set $Z:=D_1\cap D_2$ which is a smooth manifold with trivial canonical bundle. 
\end{setup}

\subsection{On the quasi-Albanese and polystability}
\label{sec:alb}
Consider the function 
\[f:M\to \mathbb C^*\]
given by $f=\frac{s_1^{d_1}}{s_2^{d_2}}$ as well as the closed holomorphic one-form $\frac{df}{f}$. Recall from Corollary~\ref{cor pi1}, that $H_1(M, \mathbb Z)/{\rm tor} \simeq \mathbb Z$ and let $\gamma$ be a generator of the latter, which can be represented by a loop based at a fixed point $x_0\in M$. Up to rescaling $s_1$, we can assume that $f(x_0)=1$. 
We set $\tau:=\int_\gamma \frac{df}{f}$ and consider the exponential map ${\bf e}:=\exp(2\pi i \cdot) : \C\to \C^*$. We introduce the fiber product 
\[\widehat M:=M\times_{\mathbb C^*}\mathbb C=\{(x,w)\in M \times \mathbb C; \,\, \be(w) = f(x)\}\]
 which sits in the commutative diagram
\[
\begin{tikzcd}
\widehat M \arrow[r, "{\widehat f}"] \arrow[d, "\pi", swap] & \mathbb C \arrow[d,"\be"]  \\
M  \arrow[r, "f"] & \mathbb C^*
\end{tikzcd}
\]
where $g$ is induced by the second projection. It is not hard to check that $\widehat M$ is connected, the map $ \pi$ induces a Galois cover with Galois group $\mathbb Z$ and that $\pi_1(\widehat M)$ is finite, cf Corollary~\ref{cor pi1}. Another way to think of the map $\widehat f$ is as follows. The closed holomorphic one-form $\frac{df}{f}$ becomes exact once pulled back to $\widehat M$, hence it is of the form $2\pi i \, d\widehat f$ as $2\pi i \, \widehat f=\pi^* \log f$ informally.\\

The following lemma identifies $f$ with the quasi-Albanese map of the pair $(X,D)$ (or $M$). We refer to \cite{FujinoQA} for details on the construction and properties of quasi-Albanese maps. 

\begin{lem}
\label{qalb}
The map $f$ can be identified to the quasi-Albanese map 
\[M\longrightarrow H^0(X, \Omega^1_X(\log D))^\vee/H_1(M, \mathbb Z)\]
of the pair $(X,D)$. 
\end{lem}

\begin{proof}
By Corollary~\ref{cor pi1}, we have $H^1(M, \mathbb C) \simeq \mathbb C$. By a result of Deligne, we have a non-canonical isomorphism
\begin{equation}
\label{deligne}
H^1(M,\mathbb C) \simeq H^0(X,\Omega^1_X(\log D))\oplus H^1(X,\mathcal O_X)
\end{equation}
and we have $H^1(X,\mathcal O_X)=0$ by Kodaira's vanishing since $X$ is Fano. Therefore, the vector space $H^0(X, \Omega^1_X(\log D))$ is generated by $\alpha :=\frac{df}{f}$.

Set $V:=H^0(X,\Omega^1_X(\log D))^\vee \simeq \mathbb C$ and $\Lambda:=H_1(M, \mathbb Z)/{\rm tor}\simeq \mathbb Z$ under the identification of $\gamma\in \Lambda$ and $1\in \mathbb Z$.  The action of $\Lambda$ on $V$ is given by $(m\gamma)\cdot v:=m\tau v$ for $m\in \mathbb Z$. Let 
\[\mathrm{alb}_{M}:M\longrightarrow V/\Lambda\]
be the quasi-Albanese map of the pair $(X,D)$. 
It is elementary to check that we have a commutative diagram
\[
\begin{tikzcd}
\widehat M \arrow[r, "\widehat{\mathrm{alb}}"] \arrow[d, "\pi", swap] & V \arrow[d, "p"]  \\
M  \arrow[r, "\mathrm{alb}_{M}"] & V/\Lambda
\end{tikzcd}
\]
where $\widehat{\mathrm{alb}}$ is defined by $\widehat{\mathrm{alb}}(\widehat x)(\alpha)=\int_{\widehat x_0}^{\widehat x}\pi^*\alpha$, where $\widehat x_0\in \widehat M$ is any fixed point lying over $x_0$, and the integral is taken along any path connecting $\widehat x_0$ to $\widehat x$. 

Let us fix an isomorphism $\psi:V\to \mathbb C$ which sends $\alpha^\vee$ to $1/\tau$ so that we have a commutative diagram
\[
\begin{tikzcd}
V \arrow[r, "\psi"] \arrow[d, "p", swap] &\mathbb C \arrow[d, "\be"]\\
V/\Lambda  \arrow[r, "\overline \psi"] & \mathbb C^*
\end{tikzcd}
\]
and $\overline \psi$ is an isomorphism. The lemma will be proved once we will have checked that 
\[\overline \psi (\mathrm{alb}_{M}(x))=f(x)\]
holds for any $x\in  M$. It is equivalent to check that 
\begin{equation}
\label{compatible} 
\be(\psi (\widehat{\mathrm{alb}}(\widehat x)))=\be(\widehat f(\widehat x)), \quad \forall \, \widehat x\in \widehat M.
\end{equation}
By construction, we have $\pi^*\alpha= 2\pi i \,d\widehat f$ hence $\psi (\widehat{\mathrm{alb}}(\widehat x))= 2\pi i (\widehat f(\widehat x)-\widehat f(\widehat x_0))$. Since $f(\widehat x_0)=1$, we have $\be(\widehat f(\widehat x_0))=1$ and \eqref{compatible} follows.
\end{proof}

\medskip

In order to make the following discussion simpler and in view of the examples we will consider in Proposition~\ref{prop:alb}, we will assume until the end of Section~\ref{sec:alb} that $d_1=d_2$, i.e. $D_1 \sim D_2$. 

Set $\mu:X'\to X$ to be the blow-up of $\Delta$. We denote by $D_i'$ the strict transform of $D_i$ and by $E$ the exceptional divisor so that $\mu^*D_i=D_i'+E$. The normal bundle of $\Delta$ in $X$ is projectively trivial (i.e. it is a direct sum of the same line bundle) hence $E\simeq \mathbb P^1\times \Delta$. 
Under the open immersion $M \hookrightarrow X'$, the map $f$ can now be extended holomorphically to a map
\begin{equation}
\label{extension}
\bar f:X'\to \mathbb P^1,
\end{equation}
which can be e.g. be understood as follows. Write $\mu^*s_i=s_i'\otimes s_E$ with $s_i'$ (resp. $s_E$) cutting out $D_i'$ (resp. $E$). Then $\bar f=[s_1':s_2']$. Note the the fiber of $\bar f$ over $[0:1]$ (resp. $[1:0]$) is naturally isomorphic to $D_1'\simeq D_1$ (resp. $D_2'\simeq D_2$) as $\bar f|_{D_1'}: D_1'\to D_1$ is just the blow up of $D_1$ along the smooth hypersurface $D_1\cap D_2$, and similary for $\bar f|_{D_2'}: D_2'\to D_2$.   

\begin{exa}
\label{pencil}
If $n$ is odd, $X=\mathbb P^n$ and $D_1, D_2$ are general of degree $d:=\frac{n+1}{2}$, then $f$ is a Lefschetz pencil; the fiber $X_t'=f^{-1}(t)$ is smooth for $t\in \mathbb P^1$ general but there are $s=(d-1)^n\cdot (n+1)$-many values $T:=\{t_1, \ldots, t_s\}$ at which $X_t'$ becomes singular. Moreover we know (cf \cite[Corollary 2.7]{HuyCubic}) that each such singular hypersurface $X'_{t}$ for $t\in T$ has exactly one singular point (an ordinary double point) which does not lie on $E$. In other words, $M_t:=f^{-1}(t)\subset M$ has a singularity. One can also see this in an elementary way since in a neighborhood of $Z=D_1\cap D_2$, one can find holomorphic coordinates $(z_i)$ such that $M_t=(z_1+tz_2=0)$ viewing $t\in \mathbb C$ locally, hence $\bar f$ is smooth near $E$. 
\end{exa}

\begin{prop}
\label{prop:alb}
In the Example~\ref{pencil} above, the following properties hold. 
\begin{enumerate}
\item The quasi-Albanese map $f:M\to \mathbb C^*$ is not locally trivial. 
\item The logarithmic tangent sheaf $T_X(-\log D)$ is not polystable.  
\item If $n\ge 9$ and $D_i$ are very general, then two general fibers of $f$ are not birational to each other.  \end{enumerate}
Moreover, Properties $(1)$ and $(2)$ above also hold for any finite étale cover of $M$. 

\end{prop}

\begin{rem}
This is in stark contrast with what happens when $D=0$, or even for klt pairs $(X,\Delta)$ such that $c_1(K_X+\Delta)=0$, in which case the Albanese map is locally trivial \cite[Theorem~0.1]{Amb05}. We refer to \cite{BFPT} for similar phenomena in the log canonical setting. Note that $T_X(-\log D)$ is semistable though, cf e.g. \cite[Theorem~A]{GSS}.
\end{rem}

\begin{rem}
In the case of $(\mathbb P^2, L+Q)$ considered in Remark~\ref{p2 lq}, the quasi-Albanese $f$ is still singular (e.g. at the point $[0:0:1]$ but it is isotrivial as the general fiber of $f$ is the smooth quadric. Yet $T_{\mathbb P^2}(-\log(L+Q))$ is not polystable, as follows from the last step of the proof of Proposition~\ref{prop:alb}. 
\end{rem}

\begin{proof}
$(1)$ follows from the fact that $f$ has some singular fibers. 

$(2)$ The map $f$ can be extended to a surjective map $\bar f:X\setminus \Delta \to \mathbb P^1$ and is is not difficult to see that it induces a non-zero morphism of coherent sheaves $T_{X\setminus \Delta} (-\log D)\to \bar f^*T_{\mathbb P^1}(-\log B) $ where $B= \{0\}+\{\infty\}$. Since $T_{\mathbb P^1}(-\log B)$ is the trivial line bundle (generated by the vector field $z\frac{\partial}{\partial z}$ and $\Delta\subset X$ has codimension two, one can extend the map to a surjective morphism of coherent sheaves $T_{X} (-\log D)\to \mathcal O_X$. If $T_{X} (-\log D)$ were polystable, one could find a non-zero holomorphic section $\mathcal O_X \to T_{X} (-\log D)$ lifting $z\frac{\partial}{\partial z}$. This is in contradiction with Proposition~\ref{no vf}, hence shows that $T_X(-\log D)$ is not polystable. %he induced vector field would lift to the blow-up $X'$ and its holomorphic flow near a general fiber of $X'\to \mathbb P^1$ would identify nearby fibers, a contradiction with the first part of Item 2. 

$(3)$ is a consequence of the stable irrationality of a very general hypersurface of degree $d\ge \log_2(n-1)+2$ in $\mathbb P^n$, cf \cite{Schreieder} and \cite[Corollary~1.3]{Shinder}.	

\smallskip 

Let us now show the last claim. If $\nu:M'\to M$ is finite étale, we still have $H^1(M', \mathbb C)\simeq \mathbb C$ since $\pi_1(M)$ is abelian of rank one. Therefore, the quasi-Albanese of $M'$ is a map $f':M'\to \mathbb C^*$. By the universal property of the quasi-Albanese map, we have a factorization
\[
\begin{tikzcd}
 M' \arrow[r, "f'"] \arrow[d, "\nu", swap] & \mathbb C^*\arrow[d,"\sigma"]  \\
M  \arrow[r, "f"] & \mathbb C^*
\end{tikzcd}
\]
and $\sigma$ is of the form $\sigma(z)=z^m$ for some $m\in \mathbb Z$. In particular, since $\nu$ is étale and $f$ has some singular fibers, so does $f'$. Let $\overline \nu :X'\to X$ be the extension of $\nu$ and $D'=\overline{\nu}^{-1}(D)$. In general, $X'$ will be singular hence one may not apply Proposition~\ref{no vf} as in item $(1)$. Yet, using the same strategy, it is enough to show that any lift $v\in H^0( X', T_{ X'}(-\log D'))$  of $z\frac{\partial}{\partial z}$ for the fibration $f'$ has to vanish, i.e. $v=0$. Now, the contraction of the logarithmic differential $\frac{df'}{f'}$ and $v$ yields a {\it bounded} holomorphic function on the quasi-projective manifold $M'$, hence it is constant. Since $\frac{df'}{f'}$ has zeros as we explained above, it follows that $\frac{df'}{f'}(v)\equiv 0$. In particular, $v\in \mathrm{Ker}(df')$ is vertical for $f'$, hence it cannot lift $z\frac{\partial}{\partial z}$. Alternatively, one can also restrict the vertical vector field $v$ to a general fiber $M'_t$ of $f'$ and see from thanks to case $k=1$ of Proposition~\ref{no vf} that $v|_{M'_t}=0$, hence $v\equiv 0$. 
\end{proof}

\medskip

\noindent
{\it On the Jordan-Hölder filtration of $T_X(-\log D)$.}
\label{JH}

\smallskip
In the Setup~\ref{set two}, we keep assuming that $D_1\sim D_2$. With the notation of the proof of Proposition~\ref{prop:alb}, recall that we have a sheaf morphism $T_{X}(-\log D)\to \bar f^*T_{\mathbb P^1}(-\log B)\simeq \mathcal O_X$ which is well-defined over the complement of $Z$, hence extends to a global sheaf morphism $\alpha:T_{X}(-\log D)\to\mathcal O_X$ by reflexivity. We define $\mathcal F:=\mathrm{ker}(\alpha)$.

For $t\in \mathbb C^*$, set $D_t=(s_1=ts_2)$ which is a Fano manifold when $t$ is general. Then $D_1\cap D_t=D_2\cap D_t=Z$. As a divisor in $D_t$, $Z$ is a smooth anticanonical and $\mathcal F|_{D_t}$ is isomorphic to the logarithmic tangent bundle of the pair $(D_t,Z)$. By Theorem~\ref{stability}, $\mathcal F|_{D_t}$ is stable with respect to $D_1|_{D_t}$ hence $\mathcal F$ is stable with respect to $D_1$. Therefore, 
\[0\subset \mathcal F \subset T_X(-\log D)\] 
is the Jordan-Hölder filtration of the semistable bundle $T_X(-\log D)$ with respect to $D_1$ (or $D_2$, or $-K_X$). Moreover, it follows from Proposition~\ref{no vf} that $T_X(-\log D)$ is not polystable (independently of the polarization).

This argument generalizes immediately to the case of a Fano manifold $X$ with an anticanonical SNC divisor $D=D_1+\cdots+D_k$ such that $k<n$ and $D_i\sim D_1$ for all $i\ge 2$. More precisely we consider the map $(\frac{s_1}{s_k}, \ldots, \frac{s_{k-1}}{s_k}):M\to (\mathbb C^*)^{k-1}$ which extends to a rational map $X\dashrightarrow (\mathbb P^1)^{k-1}$ defined in codimension one and yielding a sheaf map $\alpha:T_{X}(-\log D)\to \mathcal O_X^{\oplus (k-1)}$ whose kernel is stable with respect to $-K_X$ as one sees by restricting to a fiber. One easily gets the Jordan-Hölder filtration of $T_X(-\log D)$ from there. 
\label{stable piece}

\subsection{Riemannian geometry of the universal cover}

In this section we examine some geometric properties of the Calabi-Yau manifolds constructed in Setup~\ref{set two}.  Our first result concerns the holonomy group.

\begin{thm}
\label{thm hol 2}
In Setup~\ref{set two}, one has $\Hol(M, \omega)=\SU(n)$.
\end{thm}

\begin{rem}
\label{bochner fails}
The Bochner principle fails in this setting, since there exists a non-zero logarithmic one-form $\frac {df}{f}$ but the latter cannot be parallel by the above theorem (cf also Claim~\ref{no p forms 2} below for a more elementary argument). 
\end{rem}

\begin{rem}
The proof of Theorem~\ref{thm hol 2} will show more generally that given a (non necessarily complete) Kähler manifold $(M,\omega)$ of dimension $n$  satisfying the following two conditions
\begin{enumerate}[label=$(\roman*)$] 
\item There exists $\Omega\in H^0(M, K_M)$ satisfying $\omega^n=i^{n^2}\Omega\wedge \overline \Omega$, and
\item For any integer $0<p<n$ and any finite étale cover $q:M'\to M$, there is no non-zero $p$-form parallel with respect to $q^*\omega$,
\end{enumerate}
then the holonomy of $(M,\omega)$ is $\SU(n)$. In particular, this would apply to the conjectural complete Calabi-Yau metric on the complement $M:=X\setminus D$ of an SNC divisor $D=\sum D_i$ in a Fano manifold $X$ where each component $D_i$ is proportional to $-K_X$, cf \cite{CTY}. 
\end{rem}

Before beginning the proof, we recall some of the basic properties of the complete Calabi-Yau metric constructed in \cite{CL}. As before, we assume for simplicity $d_1=d_2=d$; the general case $d_1\ne d_2$ can be reduced to this case by appropriately rescaling the logarithmic coordinates below. Set $L_i:=\mathcal O_X(D_i)$ for $i=1,2$ (those two line bundles are isomorphic to a fixed line bundle $\mathcal L$) and fix a smooth metric $h_{\mathcal L}$ on ${\mathcal L}|_Z$ whose Chern curvature form is the K\"ahler-Ricci flat metric $\omega_Z$ in that cohomology class. Set $E_Z:=L_1|_Z\oplus L_2|_Z$; then on the total space of $E_Z$, the respective distance functions to the zero section yield functions $r_i$ and we define $x_i:=-\log r_i$. 

Consider the model K\"ahler metric $\omega_0$ on the total space of $E_Z\to Z$ which can be written as  
\[\omega_{0}:=\sum_{i,j=1}^2 \frac{\partial^2 u}{\partial x_i \partial x_j} dx_i\wedge d^c x_j +(1+\sum_{i=1}^2 \frac{\partial u}{\partial x_i}) \omega_Z\]
 $u=u(x_1,x_2)$ is a convex function on $\mathbb R_{>0}^2$ homogeneous of degree $\frac{n+2}{n}$ in the variable $x_1+x_2$.  For the purposes of constructing complete Calabi-Yau metrics, we need $u$ to solve the optimal transport problem
 \[
 \begin{aligned}
\left(\sum_{i=1}^2 \frac{\partial u}{\partial x_i}\right)^{n-2} \det D^2u &=1 \quad \text{ in } \mathbb{R}^2_{>0}\\
\nabla u (\mathbb{R}^2_{>0}) &= \{(y_1,y_2) \in \mathbb{R}^2 : y_1+y_2>0\}
\end{aligned}
\]
The existence of such a function $u$ is established in \cite{CL} (see also \cite{CTY}).
 \\

For the time being, we work on a euclidean neighborhood $U$ of a point $p\in Z$ equipped with local coordinates. Consider a trivializing section $e$ of ${\mathcal L}|_U$ with norm $|e|_{h_{\mathcal L}}^2=e^{-\phi_{\mathcal L}}$.  We have fiber coordinates $z_i$ in $L_i$ and one can assume that $s_i=z_ie$. Then $x_i=-\log |z_i|+\frac 12 \phi_{\mathcal L}$ and

 \[dx_i \wedge d^c x_j= \frac{\sqrt{-1}}{4\pi}\partial (\log z_i) \wedge \overline{ \partial \log z_j}+\mathcal{E}\] 
 
 where $\mathcal{E}=-\frac12(d \log |z_i| \wedge d^c \phi_L+d \phi_L\wedge d^c \log |z_j|) +\frac 14 d\phi_{\mathcal L}\wedge d^c \phi_{\mathcal L}$. \\

By homogeneity, we have 
\[u\sim |x|^{1+\frac{2}{n}}, |du| \sim |x|^{\frac{2}{n}}, |D^2u| \sim |x|^{\frac{2}{n}-1}\]
where $|x|:=x_1+x_2$. In particular, the error term $\mathcal{E}$ is of strictly lower order, and hence can be ignored for the purposes of analyzing the asymptotic geometry. If we write 
\[z_i=e^{-(y_i+\sqrt{-1}\theta_i)}, \quad {\rm for} \, \, i=1,2\]  the Riemannian metric associated to $\omega_0$ is asymptotic to
\begin{equation}
\label{asymptotic g0}
g_0=\sum_{i,j=1}^2 \frac{\partial^2 u}{\partial x_i \partial x_j} (dy_i\otimes dy_j+d\theta_i\otimes d\theta_j) +(\sum_{i=1}^2 \frac{\partial u}{\partial x_i}) g_Z
\end{equation}

The main theorem of \cite{CL} produces a complete Calabi-Yau metric $g_{\rm CL}$ on $M= X\setminus (D_1+D_2)$ which is asymptotic to $g_0$ in the ``generic region" $ x_i \gg 1, C^{-1}\le \frac{x_2}{x_1}\le C$.  In the ``non-generic" regions $ x_1 \gg x_2 \geq 1$ and $x_2\gg x_1 \geq 1$, the metric $g_{\rm CL}$ is asymptotic to K\"ahler metric constructed using the Tian-Yau metrics on $D_1\setminus D_2$ and $D_2 \setminus D_1$.  Precisely, for $x_1 \gg x_2$, the metric $g_{\rm CL}$ is quasi-isometric to the metric
\begin{equation}
\label{non-generic asymp g0}
\omega_0' = x_1^{\frac{{n-2}}{n(n-1)}} \omega_{{\rm {TY}}, D_1} + (x_1)^{-\frac{(n-2)}{n}} dx_1 \wedge d^cx_1
\end{equation}
where $\omega_{{\rm TY},D_1}$ denotes the Tian-Yau metric on $D_1$.  For $x_2 \sim O(1)$, which is the primary regime of interest for us, $\omega_{{\rm TY},D_1}$ is quasi-isometric to a smooth metric.  We refer the reader to \cite[Section 4]{CL} for more details.

\begin{proof}[Proof of Theorem~\ref{thm hol 2}]
We fix $x\in M$, set $G:=\Hol_x(M, \omega)$ and an identification $V:=T_{M,x}\simeq \mathbb C^n$. Since $\Omega$ is parallel, one has $G:=\Hol(M, \omega)\subset\SU(n)$. We now consider the standard representations of $G$ and its identity component $G^\circ$ on $V$.

The following vanishing result for parallel forms of intermediate degrees on covers will be crucial in what follows; it proofs is very similar to that of Claim~\ref{no p forms}.

\begin{claim}
\label{no p forms 2}
Let $0<k<n$ be an integer and let $p :S \to M$ be a finite \'etale cover. There are no non-zero (holomorphic) $k$-forms on $S$ which are parallel with respect to $p^*\omega$. 
 \end{claim}
 
 \begin{proof}[Proof of Claim~\ref{no p forms 2}]
As in the proof of Claim~\ref{no p forms}, we extend $p$ to a finite map $\bar p:Y\to X$ between normal projective varieties and we see that the asymptotics of $p^*\omega_{\rm CL}$ near a general point of $p^{-1}(D_i)$ are the same as those of $\omega$ near a general point of $D_i$. In particular a parallel $k$-form on $Y$ extends to a holomorphic section of the reflexive sheaf $\Omega^{[k]}_Y(\log p^{-1}(D))$.  On the other hand, if $(z_1,\ldots, z_n)$ is a local system of coordinates near a point in $p^{-1}(D_1\setminus D_2)$  such that $z_1$ is a local defining function for a $p^{-1}(D_1)$, then from~\eqref{non-generic asymp g0}, we see that $p^*\omega_{\rm CL}$ is quasi-isometric to
\[
(-\log(|z_1|)^{\frac{{n-2}}{n(n-1)}}\sum_{k=2}^{n} \sqrt{-1}dz_k\wedge d\bar{z}_k + (-\log|z_1|)^{-\frac{(n-2)}{n}} \frac{dz_1 \wedge d\bar{z}_1}{|z_1|^2}.
\]
We can now argue as in the proof of Claim~\ref{no p forms}.  If $\sigma$ is a holomorphic section of  $\Omega^{[k]}_Y(\log p^{-1}(D))$ then in the local coordinates $(z_1,\ldots, z_n)$ we can write
\[
\sigma = \sum_{|J|=k-1} f_{J} \frac{dz_1}{z_1} \wedge dz_{J} + \sum_{|K|=k,\, 1 \notin K} f_{K}dz_K
\]
and so
{\small
\[
|\sigma|^2 \sim \sum_{|J|=p-1}|f_J|^2(-\log|z_1|)^{\frac{(n-2)}{n}-(k-1)\frac{{n-2}}{n(n-1)}} + \sum_{|K|=k,\, 1 \notin K} |f_{K}|^2(-\log(|z_1|)^{-k\frac{{n-2}}{n(n-1)}}
\]
}
Since $0<k<n$, and $|\sigma|^2$ is constant, it follows as in the proof of Claim~\ref{no p forms} that $\sigma =0$.
\end{proof}

\begin{claim}
\label{irreducible}
The action of $G^\circ$ on $V$ is irreducible, hence $G^\circ$ is either $\SU(n)$ or $\Sp(\frac n2)$ (in that case $n$ is even). 
\end{claim}

\begin{proof}[Proof of Claim~\ref{irreducible}]
We can decompose 
\[V:=V_0\oplus \bigoplus_{i\in I} V_i\] 
where $V_0=V^{G^\circ}$ and the $V_i$ are the irreducible factors of $G^\circ$ or dimension at least two. Since $G^\circ$ is normal in $G$, $G$ preserves $V^\circ$ and permutes the $V_i$, yielding a map $G/G^\circ \to \mathfrak S_I$. The kernel of the composition $\pi_1(M)\to G/G^\circ \to \mathfrak S_I$ has finite index hence yields a finite \'etale cover 
\[p : M_{\rm hol}\to M\]
 with the property that the holonomy $G_{M_{\rm hol}}$ of  $(M_{\rm hol}, p^*\omega)$ preserves the decomposition $V=V_0\oplus \bigoplus_{i\in I} V_i$ under the appropriate identification of $G_{M_{\rm hol}}$ as a subgroup of $G$. We refer to \cite[Proposition~7.3]{GGK} for more details on this construction. 

We claim that the action of $G_{M_{\rm hol}}$ on $V$ is irreducible, from which it will follow that the action of $G^\circ$ is irreducible as well, given how $M_{\rm hol}$ was constructed. Indeed, the decomposition into $G_{M_{\rm hol}}$-irreducible pieces $V=\bigoplus_{j\in J} W_j$ would yield a splitting $G_{M_{\rm hol}}=\prod_{j\in J} G_j$ where $G_j \subset \mathrm{U}(W_j)$ by \cite[Theorem~10.38]{Besse}. Since $G_{M_{\rm hol}}\subset \SU(V)$, the product structure of $G_{M_{\rm hol}}$ implies that $G_j \subset \SU(W_j)$. Parallel transport would then provide non-zero parallel forms of degree $\dim W_j$ on $M_{\rm hol}$. This contradicts Claim~\ref{no p forms 2} unless $|J|=1$.

The last assertion in the claim follows from Berger classification since $\omega$ is not locally symmetric (otherwise it would be flat).
\end{proof}

Thanks to Claim~\ref{irreducible}, the theorem will follow if one can exclude the case $n=2m$ even and $G^\circ=\Sp(m)$. We argue by contradiction and assume that the universal cover $\widetilde \pi:\widetilde M\to M$ satisfies $\Hol(\widetilde M, \widetilde \omega)=\Sp(m)$ where $\widetilde \omega:=\widetilde \pi^*\omega$. Then there exists a parallel symplectic two-form $\widetilde \sigma$ on $\widetilde M$, unique up to scalar multiple. This defines a character
\[\chi:\pi_1(M)\to \mathbb C^*\]
such that for any $g\in \pi_1(M)$ seen as isometry of $(\widetilde M, \widetilde\omega)$, one has $g\cdot\widetilde \sigma=\chi(g) \widetilde\sigma$.  We can normalize $\widetilde\sigma$ such that $\widetilde\sigma^m=\widetilde\Omega$ where $\widetilde \Omega=\widetilde \pi^*\Omega$ since both tensors are parallel. Since the latter is invariant under the action of $\pi_1(M)$, we infer that for any $g\in \pi_1(M)$, one has
\[\widetilde\Omega=g\cdot \widetilde\Omega= g\cdot \widetilde\sigma^m  =\chi(g)^m \cdot \widetilde\Omega\]
so that $\chi$ takes values in the finite group of $m$-th roots of unity. The kernel of $\chi$ has finite index and yields a finite \'etale cover $S\to  M$ which admits a symplectic parallel $2$-form. This contradicts Claim~\ref{no p forms 2}, which ends the proof of the theorem.
\end{proof}

Next we investigate the geometric properties of the universal cover of $X\setminus (D_1+D_2)$.  

\begin{thm}
Let $\widetilde{M}$ denote the universal cover of $M= X\setminus (D_1+D_2)$, and let $\widetilde{g}_{\rm CL}$ denote the complete Calabi-Yau metric obtained from pulling back $g_{\rm CL}$.  Then we have the following
\begin{itemize}
\item The tangent cone at infinity of $(\widetilde{M}, \widetilde{g}_{\rm CL})$ is $(\mathbb{R}^2_{>0})_{(\widehat{y}_1, \widehat{y}_2)} \times \mathbb{R}_{\tau}$ equipped with the warped product metric
\[
\sum_{i,j=1}^2 u_{ij}(\overline y) d\overline{y_i}\otimes d\overline y_j+(u_{11}+u_{22}-2u_{12})(\overline y)d \overline \tau^2
\]
\item $(\widetilde{M}, \widetilde{g}_{\rm CL})$ has volume growth
\[
{\rm vol}(B_{\widetilde{g}_{\rm CL}}(p_0, R)) \sim R^{\frac{6n}{n+2}}
\]
\end{itemize}
\end{thm}

\begin{proof}
Let us first analyze the asymptotic geometry of $(\widehat M, \pi^*g_{\rm CL})$. 

\smallskip

As before, we work in the euclidean neighborhood $U$ of a point $q\in Z$. So we give ourselves $y_1,y_2>0$ and consider the loop \[\gamma_{y,q}(t)=(e^{-(y_1+2\pi\sqrt{-1}t)}, e^{-(y_2-2\pi\sqrt{-1}t)},q), \quad t\in [0,1].\] It represents a generator of the Galois group of the cover $\pi$ and it lifts to a path $\widetilde \gamma_{y,q}(t):=(\gamma_{y,q}(t), \frac{1}{2\pi\sqrt{-1}}((y_2-y_1)+4\pi \sqrt{-1}t))$ in $\widehat M$ with different endpoints. Recall that $e^{2\pi \sqrt{-1} \widehat f(\hat x)}=\frac{s_1}{s_2}(\pi(\hat x))$ so that
\[\pi^*d(\theta_2-\theta_1)=d(2\pi \, \mathrm{Re} \widehat f)\]
becomes a globally defined exact one-form on $\widehat{M}$. Let us define the variables
\[\widehat \tau:=\pi \, \mathrm{Re} \widehat f, \quad \mbox{and} \quad \theta:=\frac12(\theta_1+\theta_2).\] 
If we set $u_{ij}:=\frac{\partial^2 u}{\partial x_i \partial x_j}$, and $\widehat g_0:=\pi^*g_0$, we get
{\small
\begin{eqnarray}
\label{asymptotic g0t}
\widehat g_0&=&\sum_{i,j=1}^2 u_{ij} dy_i\otimes dy_j+(u_{11}+u_{22}-2u_{12})d\widehat \tau^2+\\
&&+(u_{11}+u_{22}+2u_{12})d\theta^2+2(u_{22}-u_{11})d\theta\otimes d\widehat \tau +(\sum_{i=1}^2 \frac{\partial u}{\partial x_i}) g_Z. \nonumber 
\end{eqnarray}}
where the above functions $u_i, u_{ij}, y_i$ (resp. $1$-form $d\theta$) on $U$ are identified with their pullback to $\pi^{-1}(U)$. 

Let us fix a point $\widehat p_0\in \widehat M$ and denote by $p_0=\pi(\widehat p_0)$ its image by $\pi$. We have the following. 

\begin{claim}
\label{claim balls}
There exists $C>0$ such that for all $R\ge 1$, we have
\begin{equation}
\label{ball comp 0}
B_{\widehat g_0}(\widehat p_0, C^{-1}R) \subset \left( |y|+|\widehat\tau| \le R^{\frac{2n}{n+2}} \right) \subset B_{\widehat g_0}(\widehat p_0, CR).
\end{equation}
\end{claim}

\begin{proof}[Proof of Claim~\ref{claim balls}]
We first observe that $\pi( B_{\widehat g_0}(\widehat p_0, R))= B_{ g_0}( p_0, R)$. Combining the latter and \eqref{asymptotic g0}, we see that if $\widehat p \in B_{\widehat g_0}(\widehat p_0, R)$ we have necessarily $|y|(\widehat p)=O( R^{\frac{2n}{n+2}})$. Therefore, along a geodesic $\gamma$ joining $\widehat p_0$ and $\widehat p$, we have up to constant
\[\widehat g_0 \ge \inf_{q\in \gamma} |y|(q)^{\frac 2n-1} d\widehat\tau^2\ge R^{\frac{2(2-n)}{n+2}}d\widehat \tau^2\]
so that $|\widehat\tau|(\widehat p)= O(R^{\frac{2n}{n+2}})$. This shows the left inclusion in the claim. The right inclusion is easily checked as follows. Pick $\widehat p=(y_1, y_2, \widehat \tau, \theta, \underline z) \in ( |y|+|\widehat\tau| \le R^{\frac{2n}{n+2}})$ with $\underline z\in Z$. The point is that one can connect $\widehat p$ to $(0,0,0, \underline z)$ with a path of length $O(R)$, from which the claim follows. In order to do so, we concatenate the following paths defined for $t\in [0, 1]$; $\gamma_1(t)=(tR^{\frac{2n}{n+2}}, tR^{\frac{2n}{n+2}}, 0, 0, \underline z)$ which has length $O((R^{\frac{n}{n+2}})^{\frac 2n-1}R^{\frac{2n}{n+2}})=O(R)$,  $\gamma_2(t)=(R^{\frac{2n}{n+2}}, R^{\frac{2n}{n+2}}, t\widehat\tau, t\theta, \underline z)$ which has length $O((R^{\frac{n}{n+2}})^{\frac 2n-1}|\widehat\tau|)=O(R)$ and finally $\gamma_3(t)=(R^{\frac{2n}{n+2}}+t(y_1-R^{\frac{2n}{n+2}}), R^{\frac{2n}{n+2}}+t(y_2-R^{\frac{2n}{n+2}}), \widehat\tau, \theta, \underline z)$ whose length is of order $O(R)$ since $|y|\le R^{\frac{2n}{n+2}}$, as one sees similarly to for $\gamma_1$. 
\end{proof}

Set $\overline{y_i}:=R^{-\frac{2n}{n+2}}y_i$ and $\overline \tau :=R^{-\frac{2n}{n+2}} \widehat\tau$. Since $u_{ij}(y)=R^{-\frac{2(n-2)}{n+2}}u_{ij}(\overline y)$, we have 
\[R^{-2}u_{ij}(y)dy_i\otimes dy_j=u_{ij}(\overline y)d\overline{y_i}\otimes d\overline y_j, \quad R^{-2}u_{ij}(y) d\widehat\tau^2= u_{ij}(\overline y) d\overline \tau^2 \]
as well as 
{\small
\[R^{-2}u_{ij}=O(R^{-\frac{4n}{n+2}}),\quad R^{-2}u_{ij}d\widehat\tau = O(R^{-\frac{2n}{n+2}})d\overline \tau, \quad  R^{-2} u_i =O(R^{\frac{4}{n+2}-2})=o(1)\]
}
\noindent
uniformly on the ball $B_{\widehat g_0}(\widehat p_0, R)$. Therefore, the last three terms in the expression of $\widehat g_0$ given in \eqref{asymptotic g0t} are collapsed when taking the tangent cone at infinity. That is, when $R\to +\infty$, the balls $(B_{\widehat g_0}(\widehat p_0, R), R^{-2}\widehat g_0)$ converge to the ball of radius one in $(\mathbb R_{>0})^2\times \mathbb R$ with respect to the warp metric 
\begin{equation}
\label{asymp cone} \sum_{i,j=1}^2 u_{ij}(\overline y) d\overline{y_i}\otimes d\overline y_j+(u_{11}+u_{22}-2u_{12})(\overline y)d \overline \tau^2. 
\end{equation}

From the expression of $\widehat g_0$, we also see that its volume form is comparable to 
\[|y|^{\frac 2n-1} dy_1d y_2 \cdot |y|^{\frac 2n-1} d\widehat\tau d\theta \cdot |y|^{\frac {2(n-2)}n} d\mathrm{vol}_{g_Z}=dy_1d y_2 d\widehat\tau  d\theta d\mathrm{vol}_{g_Z}.\]
Combined with Claim~\ref{claim balls}, this yields  
\[\mathrm{vol}(B_{\widehat g_0}(\widehat p_0, R)) \gtrsim R^{\frac{6n}{n+2}} \quad \mbox{as} \,\, R\to +\infty.\]
Since $g_{\rm CL}$ is asymptotic to $g_0$, with polynomial decay by \cite[Theorem 5.1]{CL}, we deduce the first point, and the lower bound
\[\mathrm{vol}(B_{\widehat g_{\rm CL}}(\widehat p_0, R)) \gtrsim R^{\frac{6n}{n+2}} \quad \mbox{as} \,\, R\to +\infty\]

To see the upper bound we make the following observation.  Consider the function
\[
\widehat{\rho} := \left(\widehat x_1^2+\widehat x_2^2+\widehat\tau^2+1\right)^{\frac{n+2}{4n}}.
\]
where $\widehat x_i=\pi^*(-\log |s_i|_{h_{\mathcal L}})$. 
%This formula defines $\widehat{\rho}$ is a neighborhood of $\infty$, and we extend $\widehat{\rho}$ smoothly to all of $\widehat{M}$. 
We claim that there is a uniform constant $C$ such that
\begin{equation}
\label{distance like function}
|d\widehat{\rho}|_{g_{\rm CL}}^2 \leq C
\end{equation}
where the norm is taken with respect to the complete Calabi-Yau metric.  It suffices to prove the bound in a neighborhood of $\infty$, where it follows as in \cite[Lemma 4.14]{CL}.  The only estimate we need to observe is the bound
\[
|d\widehat\tau|_{\widehat g_{\rm CL}}^2 = O(|y|^{\frac{n-2}{n}}).
\]
Above the generic region, this estimate follows from homogeneity and equation~\eqref{asymptotic g0t}, while in the non-generic region the estimate follows from \cite[Equation (38)]{CL} and the decay estimates \cite[Theorem 5.1]{CL}.  Thanks to~\eqref{distance like function}, $\widehat{\rho}$ can be viewed as a distance-like function.  Indeed, it follows from~\eqref{distance like function} that
\[
B_{\widehat{g}_{\rm CL}}(\widehat p_0, R) \subset \{ \widehat{\rho} \leq CR\}
\]
for some constant $C$.  Now it suffices to observe that \[
\int_{B_{\widehat{g}_{\rm CL}}(\widehat{p_0}, R)} (\sqrt{-1})^{n^2}\wom \wedge \overline{\wom} \lesssim\int_{\{ \widehat{\rho} \leq CR\}} dy_1d y_2 d\widehat\tau  \sim R^{\frac{6n}{n+2}}.
\]

\bigskip

It now remains to extend the above picture for $(\widehat M, \widehat g_{\rm CL})$ to $(\widetilde M, \widetilde g_{\rm CL})$. Recall from Section~\ref{sec:covers} that there exists a finite cover $p:\widetilde M\to \widehat M$ of degree $d=\mathrm{index}(D_1)$. Given $\widetilde x_0\in \widetilde M$, set $\widehat x_0=p(\widetilde x_0)$. It is elementary to see that 
\begin{equation}
\label{ball comp}
B_{\widetilde{g}_{\rm CL}}(\widetilde x_0, R)\subset p^{-1}( B_{\widehat{g}_{\rm CL}}(\widehat x_0, R)) \subset B_{\widetilde{g}_{\rm CL}}(\widetilde x_0, R+s)
\end{equation}
where $s=\max \{d_{\widetilde{g}_{\rm CL}}(\widetilde x_0, \gamma\cdot \widetilde x_0); \, \gamma \in \mathrm{Gal}(p)\}$.  Combined with the the volume estimate $\mathrm{vol}(B_{\widehat{g}_{\rm CL}}(\widehat x_0, R))\sim R^{\frac{6n}{n+2}}$ previously established, it implies that $\mathrm{vol}(B_{\widetilde{g}_{\rm CL}}(\widetilde x_0, R))\sim R^{\frac{6n}{n+2}}$ as well. 

Set $\widetilde \tau:=p^*\widehat \tau$. It follows from \eqref{ball comp 0} and \eqref{ball comp} that $B_{\widetilde g_{\rm CL}}(\widetilde x_0, R)$ is commensurable with $\big(|y|+|\widetilde \tau| \le R^{\frac{2n}{n+2}}\big)$. By pulling back \eqref{asymptotic g0t} by $p$ and applying the same reasoning as above, one sees that the asymptotic cone of $(\widetilde M, \widetilde g_{\rm CL})$ is still given by $(\mathbb R_{>0})^2\times \mathbb R$ endowed with to the warp metric \eqref{asymp cone}.
\end{proof}

\subsection{The universal cover cannot be compactified}
\label{non qp}

In this section, we will consider two families of examples where $X=\mathbb P^n$ (for $n\ge 3$) and $D_i$ ($i=1,2$) are smooth transverse hypersurfaces whose degree add up to $n+1$ but the universal cover $\widetilde M$ of $M=\mathbb P^n\setminus D$ cannot be holomorphically compactified into a compact complex space. In both cases, $\widetilde M$ carries the complete Calabi-Yau metric constructed by Collins-Li \cite{CL} but its $n$-th homology space $H_n(\widetilde M, \mathbb Z)$ has infinite rank. This has to be compared with the four-dimensional hyperkähler example constructed by Anderson, Kronheimer and LeBrun  \cite{AKL}. 

\begin{thm}
\label{univ cover}
Let $n\ge 3$ be an integer and let $D_1, D_2$ be two smooth hypersurfaces of respective degree $d_1, d_2$ in $\mathbb P^n$ intersecting transversely such that $d_1+d_2=n+1$. Assume either 
\begin{itemize}
\item $(d_1,d_2)=(1,n)$, or
\item $n$ is odd and $d_1=d_2=\frac{n+1}{2}$.
\end{itemize}
\noindent
Then the universal cover $\widetilde M$ of $M:=\mathbb P^n\setminus (D_1\cup D_2)$ has infinite topological type. In particular, it is not biholomorphic to a Zariski open set in a compact complex manifold. 
\end{thm}

\begin{rem}
\label{rem univ cover}
Let us collect a few remarks. 
\begin{itemize}
\item The refined conclusion of the above result is that the middle homology group $H_n(\widetilde M,\mathbb Z)$ has infinite rank.  Moreover, our proof applies more generally to log Calabi-Yau pairs $(X,D)$ such that there is a fibration $f:X\setminus D\to\mathbb C^*$  where $f$ has at least fiber with an ordinary double point singularity. 
\item For $d_1=d_2=1$, we have $M\simeq \mathbb C^*\times \mathbb C^{n-1}$ hence $\widetilde M \simeq \mathbb C^n$ is a quasi-projective manifold.
\item In the second case and when $n\equiv -1 [4]$, there is an alternative, short proof which was communicated to us by J\'anos Koll\'ar, which we will explain at the end of Section~\ref{non qp}, cf Remark~\ref{rem kollar}. \end{itemize} 
\end{rem}

\begin{proof}
We treat both cases separately. \\

\noindent
{\bf 1. Case $(d_1,d_2)=(1,n)$.}

\medskip

In this case, $M$ is the complement in $\mathbb C^n$ of an affine hypersurface of degree $n$. It is elementary to derive from the proof of Lemma~\ref{lem fg} that the surjection $\mathbb Z^2\to \pi_1(M)$ induces an isomorphism in restriction to $\{(u,-u); u\in \mathbb Z^2\}\simeq \mathbb Z$. In particular, $\pi_1(M)\simeq \mathbb Z $ and $ \widetilde M \simeq \widehat M$.

Moreover, it is known that $M$ has the homotopy type of the wedge of spheres $S^1\vee S^n \vee \ldots \vee S^n$, cf \cite[Corollary~1.2]{Lib94}. In particular, $\widetilde M$ has the homotopy type of a line with a bouquet of $S^n$ attached at every integer, hence $H_n(\widetilde M, \mathbb Z)$ is not finitely generated.

\bigskip

\noindent
{\bf 2. Case $n$ odd, $d_1=d_2=\frac{n+1}{2}$.}

\medskip

{\it Quick refresher on Lefschetz fibrations.} Before going any further, let us recall some well-known constructions related to vanishing cycles on the standard Lefschetz fibration. Let $q(z):=\sum_{i=1}^n z_i^2$ be the standard quadratic form on $\mathbb C^{n}$ and set $V_t:=q^{-1}(t)$ for $t\in \mathbb C$. Inside $V_1$, one can construct a differentiable submanifold $L_1:=\{|x|^2=1, |y|^2=0\}\simeq S^{n-1}\subset V_1$ where $z_k=x_k+iy_k$. If $t\neq 0$, choose a square root $\sqrt{t}$ and set $L_t:=\sqrt{t}L_1 \subset V_t$ which is well-defined independently of the choice of the square root since $L_1$ is invariant under multiplication by $-1$.  
By adjunction, one sees easily that $\Omega:=dz_1\wedge \ldots \wedge dz_n$ induces a holomorphic trivialization $\Omega_t=\frac{\Omega}{q^*dt}$ of $K_{V_t}$ for any $t\neq 0$. Moreover, $\Omega_1|_{L_1}$ determines an orientation on $L_1$, hence $L_t$ is orientable for any $t\neq 0$. However, there is no canonical choice of an orientation for all the $L_t$'s. Given a one-dimension real submanifold $T\subset \mathbb C^*$, consider the real $n$-dimensional submanifold $L_T:=q^{-1}(T)$. It $T$ is included in a simply connected open subset of $\mathbb C^*$, then there is a well-defined square root on $T$ and we get an diffeomorphism $T\times L_1\to L_T$ given by $(t,z)\mapsto \sqrt{t}z$. In particular, $L_T$ is orientable. E.g. if $T\subset \mathbb R_{>0}$,
%(resp. $T\subset \{e^{i\theta}, \theta\in [0, \pi]\}$), 
then $\sqrt{t}\Omega|_{L_T}$ is an orientation.

\medskip

Let us now go back to the proof of the theorem and let us consider the (finite) set $S\subset \mathbb C^*$ of all singular values of $f$. Since $D_1, D_2$ are general, each fiber of $f$ over such a singular value has a single ordinary double point singularity. Using dilations of $\mathbb C^*$, one can assume that $S\subset \{|z|\ge 1\}$ and $1\in S$. Using the Möbius transformation $z\mapsto \frac{2z}{z+1}$ of $\mathbb P^1$ which sends $\{|z| \ge 1\}$ to $\{\mathrm{Re}(z)\ge 1\}$, one can further assume that 
\begin{equation}
\label{min norm}
S\cap \{|z|\le 1\}=1.
\end{equation}

\medskip

We consider the segment $[0,1]\subset \mathbb C$ which is mapped to a loop around $0$ by $\be:=e^{2i\pi \cdot}$, meeting $S$ only at its endpoints thanks to \eqref{min norm}. Let us denote by $y_0$ (resp. $y_1$) the singular point of the fiber of $g$ over $0$ (resp. $1$). 

 Fix $0<\ep \ll 1$ so that over $B_{\mathbb C}(0,2\ep)$ (resp. $B_{\mathbb C}(1,2\ep)$) the map $\widehat f$ (resp. $1-\widehat f$) is given in local holomorphic coordinates near $y_0$ (resp. $y_1$) by the quadratic form $\sum_{k=1}^n z_k^2$. As recalled above, one can use the real coordinates $z_k=x_k+iy_k$ to construct for $t\in (0,\ep]$ the closed submanifold 
 \[L_t:=\{|x|^2=t, |y|^2=0\}\simeq S^{n-1}\subset \widehat M_t\]
  and similarly for  $t\in [1-\ep, 1)$, i.e. $L_t:=\{|x|^2=1-t, |y|^2=0\}$). In particular, $L_t$ gets contracted to $L_0:=y_0$ (resp. $L_1:=y_1$) as $t\to 0$ (resp. $t\to 1$). Also, the oriented closed manifolds $\Sigma_\ep:=\bigcup_{0\le t \le \ep } L_t$ (resp. $\Sigma_{1-\ep}:=\bigcup_{1-\ep \le t \le 1 } L_t$) are diffeomorphic to the closed euclidean ball $\bar B_{\mathbb R^n}(0, \sqrt{\ep})$.\\
  
 The basic idea is to construct a non-trivial cycle by ``parallel transporting" the $S^{n-1}$ vanishing cycle around the loop, and then gluing in back to itself.  The geometric construction is depicted in the Figure~\ref{fig 1} below.  The fact that this cycle is homologically non-trivial is not obvious, and is a central point of proof.
 
 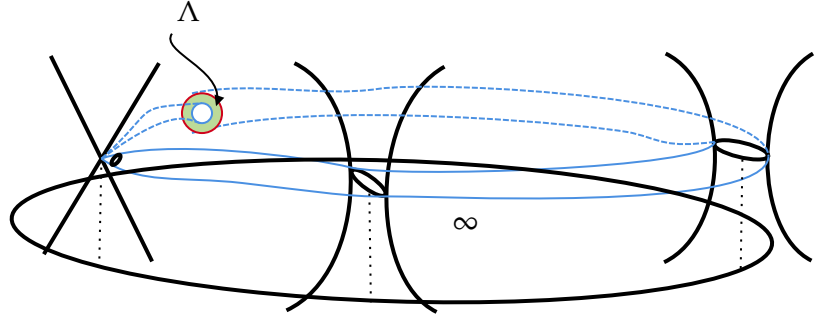
\begin{figure}[h]

\tikzset{every picture/.style={line width=0.75pt}} %set default line width to 0.75pt        

\begin{tikzpicture}[x=0.75pt,y=0.75pt,yscale=-1,xscale=1]
%uncomment if require: \path (0,354); %set diagram left start at 0, and has height of 354

%Straight Lines [id:da34182465918624116] 
\draw [line width=1.5]    (99.51,151.21) -- (150.09,254.57) ;
%Straight Lines [id:da7799788473079277] 
\draw [line width=1.5]    (153.54,154.84) -- (96.06,250.93) ;
%Straight Lines [id:da17961510915195955] 
\draw  [dash pattern={on 0.84pt off 2.51pt}]  (124.8,202.89) -- (123.65,259.01) ;
%Curve Lines [id:da7747234311023277] 
\draw [line width=1.5]    (221.13,154.84) .. controls (260.29,173.98) and (258.16,260.06) .. (217.15,279.19) ;
%Curve Lines [id:da8828107166638177] 
\draw [line width=1.5]    (296.05,156.61) .. controls (258.67,177.65) and (257.5,262.78) .. (292.09,280) ;
%Shape: Ellipse [id:dp3253279276596521] 
\draw  [line width=1.5]  (251.17,212.71) .. controls (248.46,209.29) and (249.39,207.55) .. (253.24,208.81) .. controls (257.09,210.08) and (262.4,213.88) .. (265.12,217.29) .. controls (267.83,220.71) and (266.9,222.45) .. (263.05,221.19) .. controls (259.2,219.92) and (253.88,216.12) .. (251.17,212.71) -- cycle ;
%Curve Lines [id:da9361539496608293] 
\draw [line width=1.5]    (406.43,150) .. controls (439.38,166.15) and (439.38,238.82) .. (405.66,254.97) ;
%Curve Lines [id:da6784333140992299] 
\draw [line width=1.5]    (480,150) .. controls (449.35,167.76) and (450.11,239.63) .. (479.23,254.16) ;

%Flowchart: Connector [id:dp899132651619358] 
\draw  [line width=1.5]  (454.81,198.45) .. controls (450.2,195.77) and (441.67,193.6) .. (435.74,193.6) .. controls (429.82,193.6) and (428.75,195.77) .. (433.35,198.45) .. controls (437.95,201.12) and (446.49,203.29) .. (452.41,203.29) .. controls (458.34,203.29) and (459.41,201.12) .. (454.81,198.45) -- cycle ;
%Curve Lines [id:da7711719035272134] 
\draw [color={rgb, 255:red, 74; green, 144; blue, 226 }  ,draw opacity=1 ]   (124.8,202.89) .. controls (154.31,191.99) and (227.69,201.95) .. (249.33,207.33) .. controls (270.98,212.71) and (428.66,208.14) .. (430.95,194.41) ;
%Curve Lines [id:da030928271205887037] 
\draw [color={rgb, 255:red, 74; green, 144; blue, 226 }  ,draw opacity=1 ]   (124.8,202.89) .. controls (143.58,216.21) and (164.64,211.63) .. (195.11,215.13) .. controls (225.57,218.63) and (248.79,222.41) .. (266.96,221.86) .. controls (285.13,221.32) and (457.01,230.75) .. (457.78,201.68) ;
%Curve Lines [id:da8626431368013575] 
\draw [color={rgb, 255:red, 74; green, 144; blue, 226 }  ,draw opacity=1 ] [dash pattern={on 2.25pt off 0.75pt}]  (170,190) .. controls (227,171.5) and (376,186.5) .. (390,190) .. controls (404,193.5) and (406.43,197.64) .. (430.95,194.41) ;
%Curve Lines [id:da4105803604022833] 
\draw [color={rgb, 255:red, 74; green, 144; blue, 226 }  ,draw opacity=1 ] [dash pattern={on 2.25pt off 0.75pt}]  (170,170) .. controls (200,163.5) and (233.59,172.15) .. (263.03,167.76) .. controls (292.47,163.37) and (456.24,171.8) .. (457.78,201.68) ;
%Flowchart: Connector [id:dp3289669612323848] 
\draw  [line width=1.5]  (95.71,219.54) .. controls (137.39,202.96) and (249.12,198.32) .. (345.26,209.19) .. controls (441.4,220.06) and (485.54,242.32) .. (443.86,258.9) .. controls (402.18,275.49) and (290.45,280.13) .. (194.31,269.26) .. controls (98.17,258.39) and (54.02,236.13) .. (95.71,219.54) -- cycle ;
%Shape: Ellipse [id:dp021909978573804723] 
\draw  [line width=1.5]  (131.05,202.5) .. controls (129.93,203.83) and (129.47,205.34) .. (130.04,205.87) .. controls (130.6,206.39) and (131.97,205.75) .. (133.09,204.42) .. controls (134.21,203.09) and (134.66,201.59) .. (134.1,201.06) .. controls (133.53,200.53) and (132.17,201.18) .. (131.05,202.5) -- cycle ;
%Shape: Free Drawing [id:dp8573993230955903] 
\draw  [line width=3] [line join = round][line cap = round] (295.31,237.2) .. controls (295.31,237.2) and (295.31,237.2) .. (295.31,237.2) ;
%Straight Lines [id:da24427877795444275] 
\draw  [dash pattern={on 0.84pt off 2.51pt}]  (259.29,275.16) -- (258.53,220.25) ;
%Straight Lines [id:da2679975784495223] 
\draw  [dash pattern={on 0.84pt off 2.51pt}]  (444.75,203.29) -- (443.89,258.89) ;
%Flowchart: Connector [id:dp7739411726810466] 
\draw  [color={rgb, 255:red, 208; green, 2; blue, 27 }  ,draw opacity=1 ][fill={rgb, 255:red, 188; green, 218; blue, 152 }  ,fill opacity=1 ] (165,180) .. controls (165,174.48) and (169.48,170) .. (175,170) .. controls (180.52,170) and (185,174.48) .. (185,180) .. controls (185,185.52) and (180.52,190) .. (175,190) .. controls (169.48,190) and (165,185.52) .. (165,180) -- cycle ;
%Curve Lines [id:da551502057052822] 
\draw [color={rgb, 255:red, 74; green, 144; blue, 226 }  ,draw opacity=1 ] [dash pattern={on 2.25pt off 0.75pt}]  (124.8,202.89) .. controls (154.5,176) and (138.5,177.5) .. (175,175) ;
%Flowchart: Connector [id:dp020391915860779153] 
\draw  [color={rgb, 255:red, 74; green, 144; blue, 226 }  ,draw opacity=1 ][fill={rgb, 255:red, 255; green, 255; blue, 255 }  ,fill opacity=1 ] (170,180) .. controls (170,177.24) and (172.24,175) .. (175,175) .. controls (177.76,175) and (180,177.24) .. (180,180) .. controls (180,182.76) and (177.76,185) .. (175,185) .. controls (172.24,185) and (170,182.76) .. (170,180) -- cycle ;
%Curve Lines [id:da8723790767943399] 
\draw [color={rgb, 255:red, 74; green, 144; blue, 226 }  ,draw opacity=1 ] [dash pattern={on 2.25pt off 0.75pt}]  (124.8,202.89) .. controls (164.8,172.89) and (169.5,185.5) .. (175,185) ;
%Curve Lines [id:da9572046712484884] 
\draw    (160,140) .. controls (153.86,148.98) and (184.33,157.51) .. (182.64,173.61) ;
\draw [shift={(182,176.5)}, rotate = 288.43] [fill={rgb, 255:red, 0; green, 0; blue, 0 }  ][line width=0.08]  [draw opacity=0] (5.36,-2.57) -- (0,0) -- (5.36,2.57) -- cycle    ;

% Text Node
\draw (298.69,230.63) node [anchor=north west][inner sep=0.75pt]   [align=left] {$\displaystyle \infty $};
% Text Node
\draw (161,122) node [anchor=north west][inner sep=0.75pt]   [align=left] {$\displaystyle \Lambda $};

\end{tikzpicture}
\caption{Geometric depiction of the construction of the $S^{n}$ cycle arising from parallel transporting the $S^{n-1}$ vanishing cycle around a loop containing the point $\infty$.  The shaded green region denotes the cycle $\Lambda$. }\label{fig 1}
 \end{figure}

It follows from Ehresmann theorem that the proper family $\bar f:X'\to \mathbb P^1$ defined in \eqref{extension} is $C^\infty$-trivial over the contractible base $\{\be(t), t\in [\ep, 1-\ep]\}\subset \mathbb C^*$. Moreover, one can arrange that the smooth identification of the fibers preserves the smooth hypersurface $E\subset X'$ since the latter is transverse to the fibers. Therefore the family $\widehat f: \widehat M \to \mathbb C$ is also $C^\infty$-trivial over the segment $[\ep, 1-\ep]\subset \mathbb C$. This allows us to extend to parallel transport the oriented manifold $L_{\ep}$ along $[\ep, 1-\ep]$ to obtain an oriented submanifold $\Sigma_{\ep, 1-\ep}$ with boundary diffeomorphic to a cylinder $S^{n-1}\times [\ep, 1-\ep]$. Consider 
\[L^\circ:=\Sigma_\ep \cup \Sigma_{\ep, 1-\ep}\cup \Sigma_{1-\ep}\] which is topologically the concatenation of two closed balls $\bar B_{\mathbb R^n}(0, \sqrt{\ep})$ and a cylinder. By construction, the closed ball $\Sigma_\ep$ near $y_0$ can be glued to the the "$\ep$-end" of the cylinder  $S^{n-1}\times [\ep, 1-\ep]$ along $S^{n-1}$, but there is no reason why it would be the case with the "$(1-\ep)$-end" of the cylinder and the closed ball $\Sigma_{1-\ep}$ near $y_1$. In other words, the parallel transport of $L_\ep$ along $[\ep, 1-\ep]$ need not coincide with $L_{1-\ep}$. So defined, $L^\circ$ is naturally a singular chain but its boundary $\partial L^\circ$, which is supported on $\widehat M_{1-\ep}$ may not be zero. In the next lemma, we show that $\partial L^\circ$ is realized as the boundary of a singular chain supported in the fiber $\widehat M_{1-\ep}$.

\begin{claim}
\label{cycle}
There exists a smooth singular chain $\Lambda \subset \widehat M_{1-\ep}$ such that, up to reversing the orientation of $\Sigma_{1-\ep}$, the singular chain $L:=L^\circ + \Lambda$ defines an $n$-cycle in $\widehat M$. 
\end{claim}

\begin{proof}[Proof of Claim~\ref{cycle}]
Consider the following two paths $(\gamma_i)_{i=1, 2}$ in $\mathbb C^*$ connecting $\be(\ep)$ and $\be(1-\ep)$. The first one is $\gamma_1(t)=\be(t)$ for $t\in [\ep, 1-\ep]$ as considered before. As for the second one, consider the circle $\mathcal C$ centered at $1$ of radius $r:=|1-\be(\ep)|=|1-\be(1-\ep)|$ and choose a parametrization $(\gamma_2(t))_{t\in [0,1]}$ of $\mathcal C$ so that $(\gamma_2(t))_{t\in [\ep,1-\ep]}$ sweeps out the short half circle connecting $\be(\ep)$ and $\be(1-\ep)$ (in other words, $\mathrm{Re}(\gamma_2(t))<1$ for any $t\in [\ep, 1-\ep]$). The concatenation of $\gamma_1(t)$ and $\gamma_2(1-t)$ for $t\in [\ep, 1-\ep]$ yields a loop based at $\be(\ep)$ whose interior $U$ is contractible. The loops are represented in Figure~\ref{loops}.

\begin{figure}
\centering
\begin{tikzpicture}[x=0.75pt,y=0.75pt,yscale=-1,xscale=1]
\tikzset{every picture/.style={line width=0.75pt}} %set default line width to 0.75pt

%Shape: Circle [id:dp8936362567780365]
\draw [color={rgb, 255:red, 208; green, 2; blue, 27 } ,draw opacity=1 ] (226.03,161.51) .. controls (226.01,119.82) and (259.8,86.01) .. (301.49,85.99) .. controls (343.18,85.98) and (376.99,119.76) .. (377,161.46) .. controls (377.02,203.15) and (343.23,236.96) .. (301.54,236.97) .. controls (259.85,236.99) and (226.04,203.2) .. (226.03,161.51) -- cycle ;
%Shape: Arc [id:dp4993514335479793]
\draw [draw opacity=0] (371.21,190.55) .. controls (369.52,190.85) and (367.78,191) .. (366,191) .. controls (349.43,191) and (336,177.57) .. (336,161) .. controls (336,144.43) and (349.43,131) .. (366,131) .. controls (367.24,131) and (368.46,131.07) .. (369.66,131.22) -- (366,161) -- cycle ; \draw [color={rgb, 255:red, 74; green, 144; blue, 226 } ,draw opacity=1 ] (371.21,190.55) .. controls (369.52,190.85) and (367.78,191) .. (366,191) .. controls (349.43,191) and (336,177.57) .. (336,161) .. controls (336,144.43) and (349.43,131) .. (366,131) .. controls (367.24,131) and (368.46,131.07) .. (369.66,131.22) ;

% Text Node
\draw (300,160) node [anchor=north west][inner sep=0.75pt] [font=\tiny] {$\bullet $};
% Text Node
\draw (374,158) node [anchor=north west][inner sep=0.75pt] [font=\tiny] {$\bullet $};
% Text Node
\draw (367,128) node [anchor=north west][inner sep=0.75pt] [font=\tiny] {$\bullet $};
% Text Node
\draw (368,187) node [anchor=north west][inner sep=0.75pt] [font=\tiny] {$\bullet $};
% Text Node
\draw (380,153) node [anchor=north west][inner sep=0.75pt] [font=\footnotesize] {$1$};
% Text Node
\draw (375,122.4) node [anchor=north west][inner sep=0.75pt] [font=\footnotesize] {$e( \varepsilon )$};
% Text Node
\draw (375,185.4) node [anchor=north west][inner sep=0.75pt] [font=\footnotesize] {$e( 1-\varepsilon )$};
% Text Node
\draw (307,153.4) node [anchor=north west][inner sep=0.75pt] [font=\footnotesize] {$0$};

\draw [color={rgb, 255:red, 208; green, 2; blue, 27 } ,draw opacity=1 ]  (210,153.4) node [anchor=north west][inner sep=0.75pt] [font=\footnotesize] {$\gamma_1$};

\draw [color={rgb, 255:red, 74; green, 144; blue, 226 } ,draw opacity=1 ] (335,125) node [anchor=north west][inner sep=0.75pt] [font=\footnotesize] {$\gamma_2$};

\end{tikzpicture}
\caption{The two loops} \label{loops}
\end{figure}
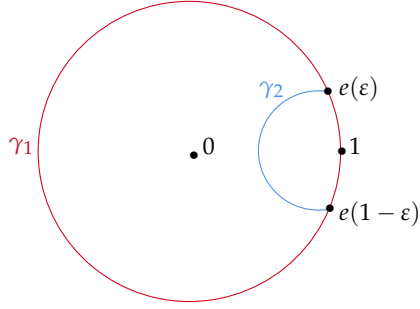

Moreover, by \eqref{min norm}, the family $\bar f$ is smooth over a small neighborhood $V$ of $\bar U$, hence it is $C^\infty$-trivial on $V$. This implies that the parallel transport operators $T_{\gamma_i}:H_n(M_{\be(\ep)}, \mathbb Z)\to  H_n(M_{\be(1-\ep)}, \mathbb Z)$ for $i=1,2$ coincide. 

Now, along $\gamma_2$, there is a well-defined square root which identify both the fibers and the vanishing cycles (as submanifolds). In particular, we have 
\[T_{\gamma_1}([L_{\be(\ep)}])=T_{\gamma_2}([L_{\be(\ep)}])=\pm [L_{\be(1-\ep)}]\quad \mbox{in} \quad H_n(M_{\be(1-\ep)}, \mathbb Z).\]
Now, $\partial L^\circ = T_{[\ep, 1-\ep]}(L_{\ep})- L_{1-\ep}\subset \widehat M_{1-\ep}$ is zero in 
$H_n(\widehat M_{1-\ep}, \mathbb Z)$ possibly after reversing the orientation on $\Sigma_{1-\ep}$. Therefore, one obtain a (smooth) singular $n$-chain $\Lambda$ in $\widehat M_{1-\ep}$ such that $\partial L^\circ = -\partial \Lambda$. 
\end{proof}

A given trivialization of $K_X+D_1+D_2$ induces a holomorphic volume form $\Omega$ on $M$; we denote by $\widehat \Omega:=\pi^*\Omega$ its pull back to $\widehat M$. By adjunction, write $\wom|_{L_t}= \wom_t \wedge dt$ and set $\Sigma^t:=L^\circ|_{\widehat M_t}$ for $t\in [0, 1]\setminus \{0, \ep, 1-\ep, 1\}$. We have
\[\int_L\wom= \int_0^1 \Big(\int_{\Sigma^t} \wom_t\Big) dt+\int_{\Lambda} \widehat \Omega.\] 
Here we implicitly identify $dt$ with its pullback by $\widehat f$. 

\begin{claim}
\label{int non zero}
There exists a holomorphic function $\tau $ on $\mathbb C$ such that the holomorphic function $\overline \tau=
\tau(\be(\widehat f)) $ on $\widehat M$ satisfies 
\[\int_L \overline \tau  \wom \neq 0.\]
\end{claim}

\begin{proof}[Proof of Claim~\ref{int non zero}]
Argue by contradiction and let $P$ be a polynomial of one variable. Set $h(t):=\int_{\Sigma^t} \wom_t$. Since $\Lambda \subset \widehat M_{1-\ep}$, we have 
\[\int_0^1 P(e^{2\pi i t})h(t)dt=-P(e^{2\pi i(1-\ep)})\int_{\Lambda} \widehat \Omega.\] 
Since this holds for any $P$, we infer that for any $1$-periodic continuous function $\lambda$ on $\mathbb R$, we have $\int_0^1\lambda(t)h(t)dt=-\lambda(1-\ep)\int_{\Lambda} \widehat \Omega$. Since $h$ is continuous away from $t=1-\ep$, this implies that $h$ vanishes identically away from $1-\ep$, which contradicts the fact that when $t\to 0$, we have $h(t)\sim \alpha t$ for some $\alpha \neq 0$  as one checks by a local computation (see, e.g. \cite[Section 4]{CollinsCIME}), given that $\wom$ is a non-vanishing holomorphic $n$-form. 
\end{proof}

From now on, we replace $\wom$ with $\overline \tau  \wom$ which now satisfies $\int_L\wom\neq 0$ and still is the pull back of a holomorphic $n$-form from $M$ (but possibly with zeros). For $k\in \mathbb Z \simeq \mathrm{Gal}(\pi)$, we denote by $L_k$ the translate of $L$ by $k$, which fibers over $[k,k+1]$ under the map $\widehat f$. 

\begin{claim}
\label{homology}
The homology classes $[L_k]\in H_n(\widehat M, \mathbb Z)$ for $k\in \mathbb Z$ are independent over $\mathbb Z$. In particular, $H_n(\widehat M, \mathbb Z)$ is not finitely generated. 
\end{claim}

\begin{proof}[Proof of Claim~\ref{homology}]
Assume by contradiction that there are non-zero integers $(a_k)_{k\in K}$ with $K\subset \mathbb Z$ finite such that
\[\sum_{k\in K} a_k [L_k]=0 \in H_n(\widehat M, \mathbb Z).\]
  For $s\in \mathbb C$, consider the (closed) holomorphic top form $e^{s\widehat f} \wom$ on $\widehat M$. Since $\wom$ is invariant under the Galois group, we have 
  \[\int_{L_k}e^{s\widehat f} \wom= e^{sk}\int_L\e^{s\widehat f}\wom \]
  hence integrating the relation between the $L_k's$ yields
  \[\Big(\sum_{k\in K}a_k e^{sk}\Big) \cdot \int_L e^{s\widehat f}\wom=0 \]
  for any $s\in \mathbb C$. Since $\int_L\wom \neq 0$, this implies that the holomorphic function $s\mapsto \sum_K a_k e^{sk}$ vanishes in the neighborhood of $0$, hence identically. Clearly, this forces $a_k=0$ for any $k\in K$. 
\end{proof}
Claim~\eqref{homology} yields that $H_{n}(\widehat M, \mathbb Z)$ has infinite rank. It remains to see that $\widetilde M$ shares the same property. This is a consequence of the universal coefficient theorem combined with the fact that the finite degree $d$ cover $p:\widetilde M\to \widehat M$ induces an surjection $H_n(\widetilde M, \mathbb R)\to H_n(\widehat M, \mathbb R)$. To check the latter assertion, first recall that $H_n(\widehat M, \mathbb R)^\vee \simeq H^n_{\rm dR, c}(\widehat M, \mathbb R)$ by Poincaré duality (and similarly for $\widetilde M$) hence it is enough to show that the pull back map $p^*:H^n_{\rm dR, c}(\widehat M, \mathbb R)\to H^n_{\rm dR, c}(\widetilde M, \mathbb R)$ is injective. Let $[\alpha] \in H^n_{\rm dR,c}(\widehat M, \mathbb R)$ be a nonzero class. By Poincaré duality again, there exists $[\beta]\in H^n_{\rm dR}(\widehat M, \mathbb R)$ such that $\int_{\widehat M} \alpha \wedge \beta \neq 0$. In particular, $\int_{\widetilde M} p^*\alpha \wedge p^*\beta = \mathrm{deg}(p) \int_{\widehat M} \alpha \wedge \beta \neq 0$ hence $p^*[\alpha]\neq 0 \in H^n_{\rm dR, c}(\widetilde M, \mathbb R)$. \\

Finally, let us explain why $\widetilde M$ is not biholomorphic to a Zariski open set of a compact complex space (or equivalently, a of a compact complex manifold thanks to Hironaka's desingularization theorem). This follows from the well-known fact that if $N$ is a Zariski open subset of a compact complex manifold $N$, then $N$ is homotopy equivalent to a finite CW complex, hence it has finitely generated (co)homology groups. One way to see it is by finding a real analytic embedding of $\overline N$ in $\mathbb C^q$ for some $q\in \mathbb N$ by Grauert's theorem \cite{Grauert58} and then use the pair version (i.e. (a semi-algebraic set, a closed subset)) of {\L}ojasiewicz's theorem on the triangulation of semi-algebraic subsets of $\mathbb R^k$, cf \cite{HironakaArcata} and references therein.  
\end{proof}

\begin{rem}[An alternative argument]
\label{rem kollar}
In the case where $n\equiv -1 [4]$ and $d_1=d_2=\frac{n+1}{2}$, one can give an alternative proof of Theorem~\ref{univ cover} as it was explained to us by J\'anos Koll\'ar. It goes as follows. First, up to perturbing the hypersurfaces, one can assume that $D_2=\overline D_1$ and that $D_1=(\sum_{i=0}^n a_i z_i^{\frac{n+1}{2}}=0)$ with $\mathrm{Re}(a_i)>0$ for all $i$. In particular, $M$ and $D$ are defined over $\mathbb R$, $D(\mathbb R)=\emptyset$ and $M(\mathbb R)\simeq \mathbb P^n(\mathbb R)$. Since $\pi_1( \mathbb P^n(\mathbb R))=\mathbb Z/2\mathbb Z$, $\Sigma:=\pi^{-1}(M(\mathbb R))\subset \widehat M$ is an infinite disjoint union of real compact manifolds of dimension $n$ either diffeomorphic to $\mathbb P^n(\mathbb R)$ or $S^n$. Moreover, $\Sigma$ arises as the fixed points of the lift of the complex conjugation to $\widehat M$. Smith's theory then implies that $H_{\bullet}(\widehat M, \mathbb Z/2\mathbb Z)$ is infinite dimensional, cf e.g. Theorem A in the lecture notes \cite{Putman}. In particular, the same holds for $\widetilde M$. 
\end{rem}

\begin{rem}
There is a unifying theme behind the arguments in the proof of Theorem~\ref{univ cover}, and the argument of Koll\'ar, recounted in Remark~\ref{rem kollar}.  Recall from the work of Harvey-Lawson \cite{HL82} that in a Calabi-Yau manifold $(X,g_{CY}, \Omega)$  a special Lagrangian submanifold $L\subset X$ is a minimal submanifold satisfying ${\rm Re}(\Omega)|_{L} = dvol_{L}$.  In particular, the homology class of a special Lagrangian is non-trivial.  In the proof of Theorem~\ref{univ cover} we construct infinitely many homologically independent cycles in the universal cover by parallel transporting the vanishing cycles of the Lefschetz fibration.  This is a commonly used technique for constructing special Lagrangians in collapsing families of $K3$ fibered Calabi-Yau $3$-folds; see \cite{ChiuLin} and the references therein.  On the other hand, the argument of Koll\'ar produces infinitely many homologically non-trivial cycles as the fixed points of a complex conjugation action.  By the same token, the fixed points of such an action are typically special Lagrangian (see e.g. \cite{Bryant99}).  Technically speaking, this conclusion depends on a suitable uniqueness statement for the Calabi-Yau metric (as in Yau's Theorem \cite{Yau78}), and such statements are not often available in the non-compact setting.
\end{rem}

\bibliographystyle{smfalpha}
\bibliography{biblio}

\end{document}